\documentclass[12pt]{amsart}

\usepackage[utf8]{inputenc}

\usepackage[margin=1in]{geometry}

\usepackage{graphicx}

\usepackage{bfc}
\usepackage{mfdmetrics}

\providecommand{\E}{\mathcal{E}}
\providecommand{\chibar}{\overline{\chi}}
\providecommand{\Etwo}{\mathbb{E}^2}

\numberwithin{theorem}{section}
\numberwithin{equation}{section}

\date{\today}
\title[The CAT(0) property for the manifold of metrics]{Geodesics,
  distance, and the CAT(0) property for the manifold of Riemannian
  metrics}
\author{Brian Clarke}
\thanks{This research was supported by NSF grant DMS-0902674.}
\address{Department of Mathematics, Stanford University, Stanford, CA
  94305-2125}
\email{\href{mailto:bfclarke@math.stanford.edu}{bfclarke@math.stanford.edu}}
\urladdr{\href{http://math.stanford.edu/~bfclarke/}{http://math.stanford.edu/${\sim}$bfclarke/}}

\begin{document}

\begin{abstract}
  Given a fixed closed manifold $M$, we exhibit an explicit formula
  for the distance function of the canonical $L^2$ Riemannian metric
  on the manifold of all smooth Riemannian metrics on $M$.
  Additionally, we examine the (metric) completion of the manifold of
  metrics with respect to the $L^2$ metric and show that there exists
  a unique minimal path between any two points.  This path is also
  given explicitly.  As an application of these formulas, we show that
  the metric completion of the manifold of metrics is a CAT(0) space.
\end{abstract}

\maketitle{}

\section{Introduction}
\label{sec:introduction}

In this paper, we give explicit formulas for the distance between
Riemannian metrics, as measured by the canonical $L^2$ Riemannian
metric on the manifold of all metrics $\M$ over a given closed
manifold $M$.  We also show that geodesics are unique and, at least on
the metric completion of the manifold of metrics, there exists a
geodesic---also explicitly given---connecting any two given points.
We then apply these formulas to give the main result of this paper
(Theorem \ref{thm:14}): that the completion of the manifold of metrics
is nonpositively curved in a metric sense.

\begin{theorem*}
  The metric completion of $\M$ with respect to its $L^2$ Riemannian
  metric, $\overline{(\M, d)}$, is a CAT(0) space.
        \end{theorem*}

Fix any closed manifold $M$ of dimension $n$, and consider the space
$\M$ of all $C^\infty$-smooth Riemannian metrics on $M$.  This space
carries a canonical weak Riemannian metric known as the $L^2$ metric
(defined in Sect.~\ref{sec:prel-manif-metr}).  The $L^2$ metric has
many interesting local and infinitesimal properties.  This local
geometry is well understood due to the work of Freed--Groisser
\cite{freed89:_basic_geomet_of_manif_of} and Gil-Medrano--Michor
\cite{gil-medrano91:_rieman_manif_of_all_rieman_metric}, who have
shown that its sectional curvature is nonpositive, and that the
geodesic equation on $\M$ is explicitly solvable.  The $L^2$ metric
has also found numerous applications, for example in the study of
moduli spaces.  Ebin \cite{ebin70:_manif_of_rieman_metric} used it to
construct a slice for the action of the diffeomorphism group on $\M$
(which thus serves a local model for the moduli space of Riemannian
metrics).  Fischer and Tromba \cite{tromba-teichmueller} have used the
$L^2$ metric to study Teichmüller spaces of Riemann surfaces, where it
naturally gives rise to the well-known Weil-Petersson metric.  Indeed,
the results of this paper stand in clear analogy with similar results
on the Weil--Petersson metric, which is geodesically convex
\cite{Wolpert:1987wv} and CAT(0) \cite{Yamada:2004wz}.

In our own work \cite{clarke:_metric_geomet_of_manif_of_rieman_metric,
  clarke:_compl_of_manif_of_rieman_metric, Clarke:2010ur},
we have focused on the global geometry of the $L^2$ metric on $\M$ and
submanifolds, studying the distance it induces between Riemannian
metrics on $M$.  This approach, in joint work with Rubinstein
\cite{clarke11:_ricci_kahler}, has led to the formulation of criteria for the
existence of Kähler--Einstein metrics and for the convergence of the
Kähler--Ricci flow on Fano manifolds.  In this paper, we deepen the
understanding of this global approach with the above-mentioned
explicit formulas for the $L^2$ distance and geodesics, as well as by
establishing the CAT(0) property for the metric completion.  (Whenever
we refer to a completion in this paper, we mean the metric
completion.)  Roughly speaking, a metric space is CAT(0) if (i)
geodesics (i.e., distance-minimizing paths) exist between any two points,
and (ii) every geodesic triangle is ``thinner'' than a triangle in the
Euclidean plane with the same side lengths, in the sense that the
edges of the triangle are closer together than in Euclidean space.
(Cf.~\cite{Bridson:1999tu}; precise definitions are also given in \S
\ref{sec:background-notation}.)  A metric space is called
nonpositively curved if it is locally CAT(0).

At this point, it is important to note that even though, as mentioned
above, the sectional curvature of the $L^2$ metric is nonpositive,
this infinitesimal result does not show that the manifold of metrics
is a nonpositively curved metric space, as it would in the
finite-dimensional setting.  Indeed, as a weak Riemannian metric on an
infinite-dimensional manifold, many theorems from finite-dimensional
Riemannian geometry fail to hold.  For instance, given any point $g_0
\in \M$, there exist other points at arbitrarily close distance to
$g_0$ that are not in the image of the exponential mapping at $g_0$.
(This last point is directly implied by Theorem \ref{thm:3} in this
paper, though it is also easy to see from the work of
Gil-Medrano--Michor
\cite[Rmk.~3.5]{gil-medrano91:_rieman_manif_of_all_rieman_metric}.)
In particular, geodesics do not necessarily exist between points in
$\M$, even locally (i.e., in a small metric ball), and so the first
criterion for $\M$ to be a nonpositively curved metric space fails.
Thus, it is important to consider the metric completion of $\M$, where
geodesics between any two points do indeed exist---note that for any
point of $\M$, there exist arbitrarily close metrics for which the
geodesic connecting the two runs through degenerate metrics, i.e.,
points not in $\M$.  (Nevertheless, as a corollary of the CAT(0)
property for the completion, we do have that geodesic triangles in
$\M$ satisfy the CAT(0) inequality, cf.~Theorem \ref{thm:14}.)

The CAT(0) property has important implications, for instance, on the
existence of generalized harmonic maps
\cite{Jost:1994ed,Jost:1995es,jost96:_gener,Jost:1997cy,jost97:_nonpos,korevaar93:_sobol,korevaar97:_global}
and on actions of groups of isometries \cite{Gelander:2008dk}.  We
plan to explore these in future work.

We wish to also briefly mention the relevance of the global approach
to the geometry of the $L^2$ metric to questions related to the
convergence of Riemannian manifolds.  In fact, Anderson
\cite{anderson92:_l_einst} has used the $L^2$ metric to study spaces
of Einstein metrics.  (He refers to the distance function of the $L^2$
metric on $\M$ as the extrinsic $L^2$ metric, because he considers the
distance obtained by infima of lengths of paths that are allowed to
travel through $\M$, as opposed to restricting paths to the submanifold
of Einstein metrics.)  One appeal of the $L^2$ metric in this context
is that it provides a very weak notion of convergence.  Another appeal
is that we have previously shown that convergence in the $L^2$ metric
implies a strong convergence of the induced measures
\cite[Cor.~4.11]{Clarke:2010ur}---which hints that it could be suited
for studying convergence of metric measure spaces.  Unfortunately,
when considered on the full space $\M$, convergence in the $L^2$
metric is perhaps \emph{too} weak---it does not imply any more
synthetic-geometric notion of convergence, such as Gromov--Hausdorff
convergence.  (For proofs and a more detailed discussion of these
facts, we refer to \cite[Sect.~4.3]{Clarke:2010ur}.)
However, as Anderson's work showed, restricted to spaces of Einstein
metrics, convergence in the $L^2$ metric in fact does imply
Gromov--Hausdorff convergence (or stronger).  An open question is what
other subspaces of $\M$ might have this desirable property.

The paper is organized as follows.  In Section
\ref{sec:preliminaries}, we set up the necessary preliminaries, both
on $L^2$ metrics on spaces of sections in general, as well as on the
$L^2$ metric on $\M$ in particular.  In Section \ref{sec:d-=-omega-1},
we find a simplified description for the $L^2$ distance between
metrics, which transforms the problem from finding the infimum of
lengths of paths in the infinite-dimensional space $\M$ into a
tractable finite-dimensional problem (Theorem \ref{thm:2}).  In
Section \ref{sec:geod-exist-uniq}, we show that there exists a unique
geodesic connecting any two given metrics in the completion of $\M$.
We also write down an explicit formula for this geodesic, which in
turn allows us to make the formula for the $L^2$ distance between
metrics explicitly computable (Theorem \ref{thm:3}).  Again, Theorem
\ref{thm:2} turns this infinite-dimensional problem into a
finite-dimensional one.  In Section \ref{sec:cat0-property}, we use
the formulas for geodesics and distance obtained in the previous
sections to show the CAT(0) property for the metric completion of
$\M$, again by extrapolation from a finite-dimensional problem.
Finally, in Section \ref{sec:outlook} we outline some open problems
regarding the $L^2$ metric that we find to be of interest.

\subsection*{Acknowledgements}
\label{sec:acknowledgements}

I would like to thank Jacob Bernstein, Guy Buss, and Michael Ka\-po\-vich
for helpful discussions during the preparation of this manuscript, as
well as Yanir A.~Rubinstein for comments on an earlier version.

\section{Preliminaries}
\label{sec:preliminaries}

\subsection{$L^2$ metrics}
\label{sec:l2-metrics}

In this subsection dealing with general $L^2$ metrics on spaces of
sections, we follow the definitions and results of Freed--Groisser
\cite[Appendix]{freed89:_basic_geomet_of_manif_of}.

Let $M$ be a smooth, closed manifold.
Let $\pi: E \rightarrow M$ be a smooth fiber bundle, and denote the
fiber over $x \in M$ by $E_x$.  Suppose we are given
\begin{enumerate}
\item a smooth volume form $\mu$ on $M$ with $\Vol(M, \mu) = 1$, and
\item a smooth Riemannian metric $g$ defined on vectors in the
  vertical tangent bundle $T^v E$.
\end{enumerate}
Let $\mathcal{E}$ denote a space of sections of $E$, where we allow the
possibilities
\begin{enumerate}
\item \label{item:6} $\mathcal{E} = \Gamma^s(E)$, the space of Sobolev
  sections $M \rightarrow E$ with $L^2$-integrable weak derivatives up
  to order $s$.  Here we require $s > n / 2$ if $E \rightarrow M$ is
  not a vector bundle.
\item \label{item:7} $\mathcal{E} = \Gamma(E)$, the space of smooth sections $M
  \rightarrow E$.
\end{enumerate}

By standard results on mapping spaces, $\E$ is a manifold in either of
these cases \cite{palais68:_found_of_global_non_linear_analy},
\cite[Example 4.1.2]{hamilton82:_inver_funct_theor_of_nash_and_moser}.
(In case \eqref{item:6}, it is a separable Hilbert manifold, and in case
\eqref{item:7}, it is a Fréchet manifold.)  With this data, we can
define an $L^2$-type Riemannian metric on $\mathcal{E}$ as follows.
The tangent space at $\sigma \in \E$ is identified with the space of
vertical vector fields ``along $\sigma$'', that is, with the space of
sections of the pulled-back bundle $\sigma^* T^v E$.  Now, for $X, Y
\in T_\sigma {\cal E}$, define the $L^2$ metric by
\begin{equation*}
  (X,Y)_\sigma := \integral{M}{}{g(\sigma(x))(X(x), Y(x))}{d \mu(x)}.
\end{equation*}
We denote by $d$ the distance function induced by $(\cdot, \cdot)$ on
${\cal E}$, and by $d_x$ the distance function induced by $g$ on
$E_x$.  Then $d$ is a pseudometric and $d_x$ is a metric (in the sense
of metric spaces).  Note that $(\cdot, \cdot)$ is in general a weak
Riemannian metric on $\E$, that is, for any $\sigma$, the topology
induced by $(\cdot, \cdot)_\sigma$ on $T_\sigma \E$ is weaker than the
manifold topology.  In this case, it is in principle possible that $d$
is not a metric in that it could fail to separate points.  There are
known examples of this due to the work of Michor--Mumford
\cite{michor06:_rieman_geomet_spaces_of_plane_curves,michor05:_vanis_geodes_distan_spaces_of},
where weak Riemannian metrics are constructed for which the induced
distance between \emph{any} two points is always zero.  However,
Theorem \ref{thm:1} below will show that $L^2$ metrics as we have
defined them do not suffer from this pathology.

Thinking of $E$ as a bundle of metric spaces $\cup_{x \in M} (E_x,
d_x)$ over $M$, we can define an $L^p$ metric $\Omega_p$ on $\E$ by
\begin{equation}
  \label{eq:25}
  \Omega_p(\sigma, \tau) := \left( \integral{M}{}{d_x(\sigma(x),
    \tau(x))^p}{d \mu(x)} \right)^{1/p}.
\end{equation}
Note that $\Omega_p$ is indeed a metric (in the sense of metric
spaces) on $\E$.  All the required properties are easily implied from
those of $d_x$.  For example, $\Omega_p$ is positive definite because,
as a Riemannian metric on a finite-dimensional manifold, $d_x$ is
positive definite for each $x$, and two unequal elements $\sigma, \tau
\in \mathcal{E}$ necessarily differ over a set of positive
$\mu$-measure.  Only the triangle inequality is not immediately
obvious---but this inequality follows, as in the case of an $L^p$
norm, from the triangle inequality for $d_x$ and Hölder's inequality.

The following theorem gives a positive lower bound for the distance,
with respect to $d$, between distinct elements of $\E$.  In the proof,
and throughout the rest of the paper, a prime will denote the partial
derivative in the variable $t$.

\begin{theorem}\label{thm:1}
  The following inequality holds for any path $\sigma_t$, $t \in [0,1]$,
  in ${\cal E}$:
  \begin{equation}\label{eq:18}
    L_{(\cdot, \cdot)}(\sigma_t)^2 \geq \integral{M}{}{L_{g}(\sigma_t(x))^2}{d \mu},
  \end{equation}
  where on the left-hand side, we measure the length in $\E$ with
  respect to $(\cdot, \cdot)$, while on the right-hand side, we
  measure the length in $E_x$ with respect to $g$.  In particular,
                for any $\sigma, \tau \in {\cal E}$, we have $d(\sigma, \tau) \geq
  \Omega_2(\sigma, \tau)$, and so $d$ is a metric on $\E$.
    \end{theorem}
\begin{proof}
  Without loss of generality, suppose that $\sigma_t$ is parametrized
  proportionally to $(\cdot, \cdot)$-arc length.  In this case, we
  have $L_{(\cdot, \cdot)}(\sigma_t)^2 = E_{(\cdot,
    \cdot)}(\sigma_t)$, where $E_{(\cdot, \cdot)}$ denotes the energy
  of the path with respect to $(\cdot, \cdot)$.  On the other hand, we have
  \begin{align*}
    E_{(\cdot, \cdot)}(\sigma_t) &=
    \integral{0}{1}{\integral{M}{}{g(\sigma_t(x))(\sigma'_t(x),
        \sigma'_t(x))}{d \mu}}{d t} =
    \integral{M}{}{\integral{0}{1}{g(\sigma_t(x))(\sigma'_t(x),
        \sigma'_t(x))}{d t}}{d \mu} \\
    &= \integral{M}{}{E_{g}(\sigma_t(x))}{d \mu} \geq
    \integral{M}{}{L_{g}(\sigma_t(x))^2}{d \mu}.
  \end{align*}
  Here, we have used Fubini's theorem followed by a well-known
  application of Hölder's inequality which gives $E_g(\sigma_t(x))
  \geq L_g(\sigma_t(x))^2$ for any $x \in M$.  As above, $E_g$ denotes
  the energy of the path with respect to $g$.  This proves
  \eqref{eq:18}, from which $d(\sigma, \tau) \geq \Omega_2(\sigma,
  \tau)$ follows directly.
      \end{proof}

Now that we have set up the situation for a general $L^2$ metric, we
turn to the main focus of this paper, when $\E$ is the space of smooth
Riemannian metrics.

\subsection{Preliminaries on the manifold of metrics}
\label{sec:prel-manif-metr}

For any point $x$ in our closed base manifold $M$, let $\satx := S^2
T^*_x M$ denote the vector space of symmetric $(0,2)$-tensors based at
$x$, and let $\s := \Gamma(S^2 T^* M)$ denote the space of smooth,
symmetric $(0,2)$-tensor fields.  Similarly, denote by $\Matx := S^2_+
T^*_x M$ the vector space of positive-definite, symmetric
$(0,2)$-tensors at $x$, and by $\M := \Gamma(S^2_+ T^* M)$ the space
of smooth sections of this bundle.  Thus, $\M$ is the space of smooth
Riemannian metrics on $M$.  In the notation of the previous section,
we have $E = S^2_+ T^* M$, $E_x = \Matx$, and $\E = \M$.  Thus we see
that $\M$ is a Fréchet manifold, and since $\M$ is an open subset of
$\s$, we have a canonical identification of the tangent space $T_g \M$
with $\s$ for any $g \in \M$.  (Similarly, the tangent space to
$\Matx$ at any $a \in \Matx$ is identified with $\satx$; thus we have
$T^v_a (S^2_+ T^* M) \cong \satx$.)

Any element $\tilde{g} \in \M$ gives rise to a natural scalar product
on $T_{\tilde{g}} \M \cong \s$ as follows.  For $h, k \in \s$, the
canonical scalar product that $\tilde{g}$ induces on $(0,2)$-tensors
is
\begin{equation*}
  \tr_{\tilde{g}}(hk) = \tr(\tilde{g}^{-1} h \tilde{g}^{-1} k) =
  \tilde{g}^{ij} h_{il} \tilde{g}^{lm} k_{jm},
\end{equation*}
where by expressions like $\tilde{g}^{-1} h$ we of course mean the
$(1,1)$-tensor obtained by raising an index of $h$ using $\tilde{g}$.
Then $\tr_{\tilde{g}}(hk)$ is a function on $M$, and by integrating it
with respect to the volume form $\mu_{\tilde{g}}$ of $\tilde{g}$, we
get a scalar product
\begin{equation}\label{eq:26}
  (h, k)_{\tilde{g}} := \integral{M}{}{\tr_{\tilde{g}}(hk)}{d \mu_{\tilde{g}}}.
\end{equation}

This $L^2$ scalar product fits into the framework of the last
subsection as follows.  For the rest of the paper, we fix some
arbitrary reference metric $g \in \M$ that has total volume $\Vol(M,
g) = 1$.  Given a tensor field $h \in \s$ or a tensor $b \in \satx$,
denote by the capital letter the $(1,1)$-tensor obtained by raising an
index using $g$, i.e., $H = g^{-1} h$ and $B = g(x)^{-1} b$.  For each
$x \in M$ and $a \in \Matx$, define a scalar product on $T_a \Matx$
(vertical vectors) by
\begin{equation*}
  \langle b, c \rangle_a := \tr_a(bc) \sqrt{\det A},
\end{equation*}
where $b, c \in T_a \Matx$.  Thus, $\langle \cdot, \cdot \rangle$
gives a Riemannian metric on $\Matx$.  For the remainder of the paper,
we denote by $\mu := \mu_g$ the volume form of $g$.  Then the scalar
product \eqref{eq:26} is given by the $L^2$ metric (in the sense of
the last section)
\begin{equation*}
  (h, k)_{\tilde{g}} = \integral{M}{}{\langle h(x), k(x)
    \rangle_{\tilde{g}(x)}}{d \mu}.
\end{equation*}
As in the last subsection, we denote by $d$ and $d_x$ the distance
functions of $(\cdot, \cdot)$ and $\langle \cdot, \cdot \rangle$,
respectively. By Theorem \ref{thm:1}, it is immediate that $d$ is a
metric on $\M$, a fact that we already proved in a less elegant way in
\cite[Thm.~18]{clarke:_metric_geomet_of_manif_of_rieman_metric}.

For $\tilde{g} \in \M$ and $a \in \Matx$, we will denote the norms
associated with $(\cdot, \cdot)_{\tilde{g}}$ and $\langle \cdot, \cdot
\rangle_a$ by $\normdot_{\tilde{g}}$ and $\absdot_a$, respectively,
throughout the remainder of the paper.

In \cite{clarke:_compl_of_manif_of_rieman_metric}, we determined the
completion of $(\M, d)$, which we will denote in the following by
$\overline{\M}$.  We will summarize the relevant details of this here.

Let $\tilde{g} : M \rightarrow S^2 T^* M$ be any measurable section
that induces a positive \emph{semi}definite scalar product on each
tangent space of $M$.  We call such a section a \emph{measurable
  semimetric}.  A measurable semimetric induces a measurable volume
form (and hence a measure) on $M$ using the usual formula
$\mu_{\tilde{g}} := \sqrt{\det \tilde{g}} \, dx^1 \wedge \cdots \wedge
dx^n$ in local coordinates.  We denote by $\Mf$ the set of all
measurable semimetrics on $M$ that have finite volume, i.e., with
$\integral{M}{}{}{d \mu_{\tilde{g}}} < \infty$.  We also introduce an
equivalence relation on $\Mf$ by saying $g_0 \sim g_1$ if and only if
the following statement holds almost surely (with respect to the
measure $\mu$) on $M$: $g_0(x)$ fails to be positive definite if and
only if $g_1(x)$ fails to be positive definite.  We then have the
following theorem.

\begin{theorem}[{\cite[Thm.~5.17]{clarke:_compl_of_manif_of_rieman_metric}}]\label{thm:5}
  There is a natural bijection between $\overline{\M}$ and $\Mfhat :=
  \Mf / {\sim}$.
\end{theorem}

In the following, we will make use of some consequences of this
theorem that we have worked out in previous papers.  Though it will
not play a direct role here, for completeness we describe the
bijection mentioned in Theorem \ref{thm:5}.  This requires a
definition.

\begin{definition}\label{dfn:6}
  Let $\{g_k\}$ be a sequence in $\M$, and let $[g_0] \in \Mfhat$.
  Define
  \begin{equation*}
    X_{g_0} := \setdef{x \in M}{\mu_{g_0}(x) = 0} \quad
    \textnormal{and} \quad D_{\seq{g_k}} := \setdef{x \in M}{\lim_{k
        \rightarrow \infty} \mu_{g_k} = 0}.
  \end{equation*}
  We say that $\{g_k\}$ \emph{$\omega$-converges} to $[g_0]$ if for
  every representative $g_0 \in [g_0]$, the following holds:
  \begin{enumerate}
  \item $\{g_k\}$ is $d$-Cauchy,
  \item $X_{g_0}$ and $D_{\{g_k\}}$ differ at most by a $\mu$-nullset,
  \item $g_k(x) \rightarrow g_0(x)$ for $\mu$-a.e.~$x \in M \setminus
    D_{\{g_k\}}$, and
  \item $\sum_{k=1}^\infty d(g_k, g_{k+1}) < \infty$.
  \end{enumerate}
  We call $[g_0]$ the \emph{$\omega$-limit} of the sequence $\{g_k\}$.
\end{definition}

The bijection of Theorem \ref{thm:5} is given by showing that: (i) For
any Cauchy sequence $\seq{g_k} \subset \M$, there exists an
$\omega$-convergent subsequence; (ii) Two $\omega$-convergent
subsequences $\seq{g^0_k}$ and $\seq{g^1_k}$ have the same
$\omega$-limit if and only if they represent the same point in
$\overline{\M}$, i.e., if and only if $\lim_{k \rightarrow \infty}
d(g^0_k, g^1_k) = 0$; and (iii) For each element $[g_0] \in \Mfhat$,
there exists a sequence in $\M$ $\omega$-converging to $[g_0]$.

At this point, we would like to point out that we will retain the
notation $d$ for the metric induced on the completion $\overline{\M}$
from $(\M, d)$.  It will also be convenient to use the bijection of
Theorem \ref{thm:5} to see $d$ as a metric on $\Mfhat$, and as a
\emph{pseudo}metric on $\Mf$.  Of course, for $g_0, g_1 \in \Mf$, we
have $g_0 \sim g_1$ if and only if $d(g_0, g_1) = 0$.

In what follows, we will also be concerned with special subsets of
$\M$ that have convenient properties.  They are essentially subsets
that are, in a pointwise sense, uniformly bounded away from infinity
and the boundary of $\M$.

\begin{definition}\label{dfn:1}
  For $\tilde{g} \in \M$ and $x \in M$, let
  $\lambda^{\tilde{G}}_{\min}(x)$ denote the minimal eigenvalue of
  $\tilde{G}(x) = g(x)^{-1} \tilde{g}(x)$.  A subset $\U \subset \M$
  is called \emph{amenable} if it is of the form
  \begin{equation}\label{eq:24}
    \U = \{ \tilde{g} \in \M \mid \lambda^{\tilde{G}}_{\min}(x) \geq \zeta\
    \textnormal{and}\ |\tilde{g}(x)|_{g(x)} \leq C\ \textnormal{for all $\tilde{g}
      \in \U$ and $x \in M$} \}
  \end{equation}
  for some constants $C, \zeta > 0$.

  We denote the closure of $\U$ in the $L^2$ \emph{norm} $\normdot_g$
  by $\U^0$; it consists of all measurable, symmetric $(0, 2)$-tensors
  $\tilde{g}$ satisfying the bounds of \eqref{eq:24} $\mu$-a.e.
\end{definition}

\begin{remark}\label{rmk:1}
  \
  
  \begin{enumerate}
  \item \label{item:4} Note that the preceding definition differs from that in our
    previous works
    (cf.~\cite[Def.~3.1]{clarke:_compl_of_manif_of_rieman_metric},
    \cite[Def.~2.11]{Clarke:2010ur}).  The above definition
    is coordinate-independent and therefore more satisfying.
    Additionally, the results we need from those previous works are
    valid for the definition here because of the following
    equivalence:  If $\U \subset \M$ is amenable in this new sense, then there
    exist $\U', \U'' \subset \M$ that are amenable in the old sense,
    and such that $\U' \subset \U \subset \U''$.
  \item \label{item:5} If $\U \subset \M$ is amenable, then $\U^0$ is
    \emph{pointwise convex}, by which we mean the following.  Let
    $g_0, g_1 \in \U^0$, and let $\rho$ be any measurable function on
    $M$ taking values between $0$ and $1$.  Then $\rho g_0 + (1 -
    \rho) g_1 \in \U^0$.  This is straightforward to see by the
    concavity of the function mapping a matrix to its minimal
    eigenvalue, and the convexity of the norm $\absdot_{g(x)}$.
  \end{enumerate}
\end{remark}

The following lemma was originally proved in
\cite[Lem.~3.3]{clarke:_compl_of_manif_of_rieman_metric} for amenable
subsets, but the same proof (which is more or less self-evident) works
for $L^2$ closures of amenable subsets.

\begin{lemma}\label{lem:10}
  Let $\U$ be an amenable subset.  Then there exists a constant $K >
  0$ such that for all $\tilde{g} \in \U^0$,
  \begin{equation}\label{eq:130}
    \frac{1}{K} \leq \left( \frac{\mu_{\tilde{g}}}{\mu} \right) \leq K,
  \end{equation}
  where by $(\mu_{\tilde{g}} / \mu)$ we denote the unique measurable
  function on $M$ such that $\mu_{\tilde{g}} = (\mu_{\tilde{g}} / \mu)
  \mu$.
\end{lemma}

To end this subsection, we have a somewhat unexpected and extremely
useful result that bounds the distance between two semimetrics
\emph{uniformly} based on the intrinsic volume of the subset on which
they differ.

\begin{proposition}[{\cite[Prop.~2.20]{Clarke:2010ur}}]\label{prop:3}
  Let $g_0, g_1 \in \Mf$ and $A := \carr(g_1 - g_0)$.  Then
  \begin{equation*}
    d(g_0, g_1) \leq C(n)
    \left(
      \sqrt{\Vol(A, g_0)} + \sqrt{\Vol(A, g_1)}
    \right),
  \end{equation*}
  where $C(n)$ is a constant depending only on $n = \dim M$.
\end{proposition}

\section{$d = \Omega_2$ on $\M$}
\label{sec:d-=-omega-1}

In this section, we show that the distance function of the $L^2$
Riemannian metric is exactly given by the $L^2$-type metric $\Omega_2$
that we defined in \eqref{eq:25}.

\subsection{Paths of degenerate metrics and Riemannian distances}
\label{sec:paths-degen-metr}

If $g_0, g_1 \in \M$ and $g_t$ is a piecewise differentiable path in
$\M$ between them, then $d(g_0, g_1) \leq L(g_t)$.  The goal of this
subsection is to prove a similar inequality for certain paths of
semimetrics in $\Mf$.

We first have to be precise about what $L(g_t)$ should mean if $g_t
\in \Mf$.  We denote by $\MC \subset \Mf$ the set of all continuous
Riemannian metrics on $M$.  By $\sC$, we denote the closure of $\s$ in
the $C^0$ norm.  For $\tilde{g} \in \Mf$, denote by $\szg{\tilde{g}}$
the set of measurable $(0,2)$-tensor fields $h$ such that $h(x) = 0$
whenever $\tilde{g}(x)$ is not positive definite, and such that the
quantity
\begin{equation*}
  \norm{h}_{\tilde{g}} :=
    \left( \integral{M \setminus X_{\tilde{g}}}{}{\tr_{\tilde{g}}(h^2)
        \sqrt{\det \tilde{G}}}{d \mu} \right)^{1/2}
\end{equation*}
is finite, where in the above $X_{\tilde{g}} \subseteq M$ denotes the
set on which $\tilde{g}$ is not positive definite.

We will consider paths of (semi-)metrics $g_t$, $t \in [0,1]$, in both
$\MC$ and $\Mf$.  We will call such a path $g_t$ differentiable in
$\MC$ (resp.~$\Mf$) if, for each $x \in M$, $g_t(x)$ is a
differentiable path in $\Matx$ and, additionally, $g'_t$ is contained
in $\sC$ (resp.~$\szg{g_t}$) for all $t \in [0,1]$.

\begin{definition}\label{dfn:4}
  For $E \subseteq M$, we call a path $g_t$, $t \in [0,1]$, in $\Mf$
  \emph{continuous on $E$} if $g_t(x)|_E$ is continuous in $x$ for all
  $t$.  If $E = M$, we call $g_t$ simply \emph{continuous}.

  To avoid confusion, we emphasize that a continuous path is one that
  is continuous in $x$ for each $t$, and a differentiable path is one
  that is differentiable in $t$ for each $x$.
\end{definition}

Let $g_t$, $t \in [0,1]$, be a path in $\Mf$ or $\MC$ that is
piecewise differentiable.  We denote by $L(g_t)$ the length of $g_t$
as measured in the naive ``Riemannian'' way:
\begin{equation*}
  L(g_t) = \integral{0}{1}{\norm{g'_t}_{g_t}}{dt}.
\end{equation*}
When we refer to the length $L(a_t)$ of a path $a_t$ in $\Matx$, we
implicitly mean the length with respect to $\langle \cdot, \cdot
\rangle$.

It is easy to see (cf.~also the proof of
\cite[Cor.~3.16]{clarked):_compl_of_manif_of_rieman}) that the $C^0$
topology on $\MC$ is stronger than the Riemannian $L^2$ topology.  Let
$g_t$ be a piecewise differentiable path in $\MC$ connecting two
continuous metrics $g_0$ and $g_1$.  It is intuitive, but perhaps not
immediately clear, that using smooth approximations, one could show
$d(g_0, g_1) \leq L(g_t)$ as in the case of smooth metrics.  We
formalize this in the following lemma.  The proof is straightforward,
but we include some details for those readers unfamiliar with
regularization of tensors on manifolds.

\begin{lemma}\label{lem:7}
  Let $g_0, g_1 \in \MC$, and suppose that $g_t$, $t \in [0,1]$, is a
  piecewise differentiable path in $\MC$ connecting them.  Then
  $d(g_0, g_1) \leq L(g_t)$.
\end{lemma}
\begin{proof}
  Let $\{U_\alpha, \varphi_\alpha\}$ be a finite atlas of charts
  $\varphi_\alpha : U_\alpha \rightarrow \R^n$ for $M$.  Choose a
  partition of unity $p_\alpha$ subordinate to this atlas.  We denote
  the push-forward of $g_t$ via $\varphi_\alpha$ by $g_t^\alpha$ and,
  by an abuse of notation, denote the locally-defined tensor obtained
  from restricting $g_t$ to $U_\alpha$ by the same.  We will
  regularize these metrics by convolution in local coordinates,
  letting $\phi$ be any function on $\R^n$ that has norm $1$ in
  $L^1(\R^n)$ and that vanishes outside the unit ball.  Defining
  $\phi_\epsilon(x) := \epsilon^{-n} \phi(x / \epsilon)$, we have that
  for all $i$, $j$, and $\alpha$, the convolutions
  $(g_t^{\alpha,\epsilon})_{ij} := \phi_\epsilon * (g_t^\alpha)_{ij}$
  and $(g_t^{\alpha,\epsilon})'_{ij} := (\phi_\epsilon *
  g_t^\alpha)'_{ij} = \phi_\epsilon * (g_t^\alpha)'_{ij}$ (the prime,
  as usual, denotes the partial derivative w.r.t.~$t$) are smooth
  functions converging in the $C^0$ norm to $(g_t^\alpha)_{ij}$ and
  $(g_t^\alpha)'_{ij}$, respectively, as $\epsilon \rightarrow 0$.
  Furthermore, since we are dealing with a finite number of indices,
  for any given $t$ we can choose $\epsilon > 0$ small enough that
  $(g_t^{\alpha,\epsilon})_{ij}$ and $(g_t^{\alpha,\epsilon})'_{ij}$
  are uniformly $C^0$-close to $g_t^\alpha$ and $(g_t^\alpha)'$,
  respectively.

  One easily sees that for each $\epsilon > 0$,
  $\norm{\phi_\epsilon}_{L^1(\R^n)} = 1$, and also that this implies
  that convolution with $\phi_\epsilon$ has operator norm $1$ when
  viewed as a linear operator on $C^0(\R^n)$.  From this, and the
  compactness of the time interval on which $g_t$ is defined, one can
  straightforwardly conclude that $\epsilon > 0$ can be chosen small
  enough that $(g_t^{\alpha,\epsilon})_{ij}$ and
  $(g_t^{\alpha,\epsilon})'_{ij}$ are uniformly $C^0$-close to
  $g_t^\alpha$ and $(g_t^\alpha)'$, respectively, independently not
  just of $i$, $j$, and $\alpha$, but also $t$.

  Finally, using our partition of unity, we define $g_t^\epsilon :=
  \sum_\alpha p_\alpha g_t^{\alpha,\epsilon}$ to get a path of smooth
  Riemannian metrics connecting $g_0^\epsilon$ and $g_1^\epsilon$.  We
  have $(g_t^\epsilon)' = \sum_\alpha p_\alpha
  (g_t^{\alpha,\epsilon})'$, and so one sees that with $\epsilon$
  small enough, $L(g_t^\epsilon)$ is arbitrarily close to $L(g_t)$.
  By the above-mentioned fact that the $C^0$ topology on $\M$ is
  stronger than the Riemannian $L^2$ topology, we also have that
  $d(g_0, g_0^\epsilon)$ and $d(g_1, g_1^\epsilon)$ can be made
  arbitrarily small, from which the desired result follows.
\end{proof}

Using the above result on paths of continuous metrics, we can prove
what we need about paths in $\Mf$.  We first briefly set up some
notation, and then state the result in a lemma.

\begin{definition}\label{dfn:3}
  Let $E \subseteq M$ be any subset.  We denote by $\chi(E)$ the
  characteristic (or indicator) function of $E$.  The characteristic
  function of its complement is denoted by $\chibar(E) := \chi(M
  \setminus E)$.
\end{definition}

\begin{lemma}\label{lem:2}
  Let $g_0, g_1 \in \MC$, and let $g_t$, $t \in [0,1]$, be any smooth
  path in $\MC$ from $g_0$ to $g_1$.  Furthermore, let $E
  \subseteq M$ be any measurable subset.

  We define $\tilde{g}_t := \chibar(E) g_0 + \chi(E) g_t$; in
  particular $\tilde{g}_1 = \chibar(E) g_0 + \chi(E) g_1$.
  Then
  \begin{equation*}
    d(g_0, \tilde{g}_1) \leq L(\tilde{g}_t).
  \end{equation*}
\end{lemma}
\begin{proof}
  For each $k \in \N$, choose an open set $U_k$ and a closed set $Z_k$
  such that $Z_k \subseteq E \subseteq U_k$, and such that $\mu(U_k
  \setminus Z_k) < \frac{1}{k}$.  (This is possible because the
  Lebesgue measure is regular.)  We also choose continuous functions $f_k$
  with the properties that
  \begin{enumerate}
  \item \label{item:1} $0 \leq f_k \leq 1$,
  \item \label{item:2} if $x \not\in U_k$, then $f_k(x) = 0$, and
  \item \label{item:3} if $x \in Z_k$, then $f_k(x) = 1$.
  \end{enumerate}

  For each $t \in [0,1]$, let $\tilde{g}^k_t := f_k g_t + (1 - f_k)
  g_0$.  Thus, we have that in particular, $\tilde{g}^k_1$ coincides
  with $g_1$ on $Z_k$ and with $g_0$ on $M \setminus U_k$.  Our goal
  is to show that $\lim_{k \rightarrow \infty} L(\tilde{g}_t^k) \leq
  L(\tilde{g}_t)$ and $d(\tilde{g}_1^k, \tilde{g}_1) \rightarrow 0$ as
  $k \rightarrow \infty$, as we can then conclude from the triangle
  inequality and Lemma \ref{lem:7} that
  \begin{equation*}
    d(g_0, \tilde{g}_1) \leq d(g_0, \tilde{g}_1^k) + d(\tilde{g}_1^k,
    \tilde{g}_1) \leq L(\tilde{g}_t^k) + d(\tilde{g}_1^k,
    \tilde{g}_1).
  \end{equation*}
  The statement of the lemma then follows by passing to the limit on
  the right.

  We begin with the claim that $\lim_{k \rightarrow \infty}
  L(\tilde{g}_t^k) \leq L(\tilde{g}_t)$.  Since $M$ and $[0,1]$ are
  compact, we have that
  \begin{equation*}
    N := \max_{x \in M, t \in [0,1]} \abs{g'_t(x)}_{g_t(x)}^2 < \infty.
  \end{equation*}
  (Recall that by the definition of a smooth path in $\MC$, we have $g'_t \in
  \sC$ for all $t$.)  Therefore, noting that $(\tilde{g}_t^k)' = g'_t$
  on $Z_k$ and $(\tilde{g}_t^k)' \equiv 0$ on $M \setminus U_k$, we
  may estimate
  \begin{align*}
    \norm{(\tilde{g}_t^k)'}^2_{\tilde{g}_t^k} &=
    \integral{Z_k}{}{\abs{g'_t}_{g_t}^2}{d \mu} + \integral{U_k
      \setminus Z_k}{}{\abs{f_k g'_t}_{f_k g_t}^2}{d \mu} \\
    &= \norm{\chi(Z_k) g'_t}_{g_t}^2 + \integral{U_k \setminus
      Z_k}{}{\tr_{f_k g_t} ((f_k g'_t)^2) \sqrt{\det (f_k
        G_t)}}{d \mu} \\
    &= \norm{\chi(Z_k) g'_t}_{g_t}^2 + \integral{U_k
      \setminus Z_k}{}{f_k^{n/2} \tr_{g_t} ((g'_t)^2) \sqrt{\det
        G_t}}{d \mu} \\
    &\leq \norm{\chi(E) g'_t}_{g_t}^2 + \integral{U_k
      \setminus Z_k}{}{N}{d \mu}.
  \end{align*}
  The first term in the last line is just
  $\norm{\tilde{g}'_t}_{\tilde{g}_t}^2$, since $\tilde{g}_t = g_t$ on
  $E$.  The second term is just $N \cdot \mu(U_k \setminus Z_k) < N /
  k$, which converges to zero as $k \rightarrow \infty$.  Thus, we
  have that $\norm{\tilde{g}_t^k}_{\tilde{g}_t^k} \leq
  \norm{\tilde{g}_t}_{\tilde{g}_t} + N/k$ for each $t \in [0,1]$,
  which implies the claim that $\lim_{k \rightarrow \infty}
  L(\tilde{g}_t^k) \leq L(\tilde{g}_t)$.

  We now move on to the claim that $\lim_{k \rightarrow \infty}
  d(\tilde{g}_1^k, \tilde{g}_1) = 0$.  Since $g_0$ and $g_1$ are
  continuous metrics, it is clear that we can find an amenable subset
  $\U$ such that $g_0, g_1 \in \U^0$.  But we also know that at each
  point, $\tilde{g}_1^k$ and $\tilde{g}_1$ are linear combinations of
  $g_0$ and $g_1$ with coefficients between zero and one.  Hence, by
  the pointwise convexity of $L^2$ closures of amenable subsets
  (cf.~Remark \ref{rmk:1}\eqref{item:5}), $\tilde{g}_1^k, \tilde{g}_1
  \in \U^0$ for all $k \in \N$.  Thus, by Lemma \ref{lem:10}, there
  exists a constant $K$ such that
  \begin{equation}\label{eq:2}
    \left(
      \frac{\mu_{\tilde{g}_1^k}}{\mu}
    \right) \leq K \ \textnormal{for all}\ k \in \N \ 
    \textnormal{and} \
    \left(
      \frac{\mu_{\tilde{g}_1}}{\mu}
    \right) \leq K.
  \end{equation}
  Using this, Proposition \ref{prop:3}, and the fact that
  $\tilde{g}_1^k$ and $\tilde{g}_1$ differ only on $U_k \setminus
  Z_k$, we can conclude
  \begin{equation*}
    d(\tilde{g}_1^k, \tilde{g}_1) \leq
    C(n)
    \left(
      \sqrt{\Vol(U_k \setminus Z_k, \tilde{g}_1^k)} + \sqrt{\Vol(U_k
        \setminus Z_k, \tilde{g}_1)}
    \right)
    \leq 2 C(n) \sqrt{\frac{K}{k}}.
  \end{equation*}
      This proves the second claim and so, as noted above, the statement
  of the lemma follows.
\end{proof}

In what follows, we will have to deal with reparametrizations of
paths.  Given a path in $\Mf$ (or in any space of sections), one can
reparametrize globally, in that one replaces $g_t$, $t \in [0,1]$,
with $g_{\varphi(t)}$, for some appropriate $\varphi : [0,1]
\rightarrow [0,1]$.  One can also ``reparametrize'' pointwise, in that
one uses $g_{\varphi_x(t)}$, where for each $x \in M$, $\varphi_x :
[0,1] \rightarrow [0,1]$ is a function with the appropriate
properties.  Of course, the latter changes the image of the path in
$\Mf$, but for our purposes it can do so in advantageous ways.  The next
definition deals with the specific reparametrizations we will need.

\begin{definition}\label{dfn:2}
  Let $g_t$, $t \in [0,1]$, be a path in $\Mf$.  By the
  \emph{pointwise reparametrization of $g_t$ proportional to arc
    length}, we mean the path in $\tilde{g}_t$ in $\Mf$, $t \in
  [0,1]$, where for each $x \in M$, $\tilde{g}_t(x)$ is the path
  obtained from $g_t(x)$ by reparametrization proportional to $\langle
  \cdot, \cdot \rangle$-arc length.
\end{definition}

Given this definition, the following lemma is essentially self-evident.

\begin{lemma}\label{lem:6}
  Let $g_0, g_1 \in \Mf$, and let $g_t$ be a piecewise differentiable
  path in $\Mf$ connecting $g_0$ and $g_1$.  Suppose $g_t$ fails to
  have a two-sided $t$-derivative at times $0 = t_0 < t_1 < \cdots <
  t_k = 1$.  If $g_t$ is continuous on $E \subseteq M$ for all $t \in
  [0,1]$, and $L(g_t(x) |_{[t_i, t_{i+1}]})$ is continuous as a
  function of $x$ on $E$ for all $i = 0, \dots, k-1$, then the path
  obtained from $g_t$ via pointwise reparametrization proportional to
  arc length is continuous on $E$.
\end{lemma}

\subsection{Proof that $d = \Omega_2$}
\label{sec:proof-d-=}

We now get into the heavy lifting of this section.  We will need two
rather technical lemmas to get from the restricted situation of Lemma
\ref{lem:2} to the desired general result.  In the following, we will
always denote by $B_\delta(x)$ the closed geodesic ball around $x \in
M$ with radius $\delta$ (w.r.t.~the fixed reference metric $g$).

\begin{lemma}\label{lem:3}
  Let any $g_0, g_1 \in \M$ and $\epsilon > 0$ be given.  Then there
  exists a $\delta = \delta(\epsilon, g_0, g_1) > 0$ with the property
  that given any $x_0 \in M$, we can find a path $g_{x_0,t}$ in $\Mf$,
  for $t \in [-\epsilon, 1+\epsilon]$, from $g_0$ to
  $\chibar(B_\delta(x_0)) g_0 + \chi(B_\delta (x_0)) g_1$ such that
  for each $x \in B_\delta(x_0)$, we have
  \begin{equation}\label{eq:4}
    \abs{g'_{x_0,t}(x)}_{g_{x_0,t}} <
    \begin{cases}
      1, & t \in [-\epsilon, 0) \cup [1, 1 + \epsilon], \\
      d_{x}(g_0(x), g_1(x)) + 3 \epsilon, & t \in [0,1).
    \end{cases}
  \end{equation}
          Furthermore, for each $t$, $g_{x_0,t}$ is constant on $M \setminus
  B_\delta(x_0)$ and is continuous on $B_\delta(x_0)$.
                \end{lemma}
\begin{proof}
  For a given $x_0 \in M$, we may choose a smooth path $a_{x_0,t}$, $t
  \in [0,1]$, in $\M_{x_0}$ connecting $g_0(x_0)$ and $g_1(x_0)$ that
  has length
  \begin{equation}\label{eq:7}
    L(a_{x_0,t}) < d_{x_0}(g_0(x_0), g_1(x_0)) + \epsilon.
  \end{equation}
  Furthermore, this path can be chosen in such a way that there exist
  constants $\zeta, \tau > 0$, depending on $g_0$, $g_1$, and
  $\epsilon$ but not on $x_0$, such that
  \begin{equation*}
    a_{x_0,t} \in \M_{x_0}^{\zeta,\tau} := \{ a \in \M_{x_0} \mid \det
    A \geq \zeta\ \textnormal{and}\ \abs{a}_{g(x_0)} \leq \tau\
    \textnormal{for all}\ 1 \leq i, j \leq n \}
  \end{equation*}
  for all $x_0 \in M$ and $t \in [0,1]$.  (This relies on the fact that
  $g_0$ and $g_1$ are smooth metrics, and so are contained in a
  common compact subset of $S^2_+ T^* M$.)

  For the rest of the proof, when we refer to geodesics and the
  Levi-Civita connection $\nabla$ on $M$, we mean those belonging to
  our fixed reference metric $g$.  Now, let $\delta$ be small enough
  that for any $x_0 \in M$ and any $x \in B_\delta(x_0)$, there exists
  a unique minimal geodesic (up to reparametrization) from $x_0$ to
  $x$.

  The Levi-Civita connection $\nabla$ can be extended to all tensor
  fields, and in particular to $T^* M \otimes T^* M$.  A brief
  calculation shows that if $h \in S^2 T^* M$ and $X$ is a vector
  field on $M$, then $\nabla_X h \in S^2 T^* M$.  (That is, symmetry
  is preserved.)  Therefore $\nabla$ induces a connection
  $\overline{\nabla}$ on the vector bundle $S^2 T^* M$.
  
  For each $x_0 \in M$ and $x \in B_{\delta}(x_0)$, we denote by
  $P_{x_0,x}$ the parallel transport with respect to
  $\overline{\nabla}$ along the minimal geodesic from $x_0$ to $x$.
  In local vector bundle coordinates for $S^2 T^* M$, the parallel
  transport of an element of $\s_{x_0} = S^2 T^*_{x_0} M$ is the
  solution of a first-order linear ODE with coefficients depending
  smoothly upon $x_0$ and $x$.  We know that $P_{x_0,x}$ is a linear
  isometry (w.r.t.~the scalar product induced by $g$) between
  $\s_{x_0}$ and $\satx$, so $P_{x_0,x}(a)$ depends smoothly on $a \in
  s_{x_0}$.  Furthermore, since solutions of ODEs behave continuously
  under perturbations of the coefficients, the mapping $(x_0, x)
  \mapsto P_{x_0,x}$ is continuous \cite[Ch.~1,
  Thm.~7.4]{coddington55:_theor}.
            
  We next let $\tilde{a}_{x_0,t}(x)$ be the path in $S^2 T^*_x M$ obtained from
  $a_{x_0,t}$ by parallel transport, i.e., $\tilde{a}_{x_0,t}(x) =
  P_{x_0,x}(a_{x_0,t})$.  By the discussion above on the
  smoothness/continuity of parallel transport, $\tilde{a}_{x_0,t}(x)$ is smooth
  in the $t$ variable and continuous in the $x$ variable.

  Now, since $(x_0, x, a) \mapsto P_{x_0,x}(a)$ is continuous, this
  mapping is uniformly continuous when restricted to the compact space
  \begin{equation}\label{eq:20}
    \bigcup_{x_0 \in M}
    \left(
      \{x_0\} \times \bigcup_{x \in B_{\delta}(x_0)} \M_x^{\zeta,\tau}
    \right).
  \end{equation}
  Thus we may (by shrinking $\delta$
  if necessary) assume that $P_{x_0,x}(\M_{x_0}^{\zeta,\tau}) \subset \M_x$
  for each $x_0 \in M$ and $x \in B_\delta(x_0)$.

  Since we have assumed that each path $a_{x_0,t}$ is contained in
  $\M_{x_0}^{\zeta,\tau}$, this implies that $\tilde{a}_{x_0,t}(x)$ is a path
  in $\M_x$ running from $P_{x_0,x}(g_0(x_0))$ to
  $P_{x_0,x}(g_1(x_0))$.  Again by continuity of parallel transport,
  by shrinking $\delta$ we may assume that
  \begin{equation}\label{eq:3}
    d_x(g_0(x), P_{x_0,x}(g_0(x_0))) \leq \eta \quad \textnormal{and}
    \quad d_x(g_1(x), P_{x_0,x}(g_1(x_0))) \leq \eta
  \end{equation}
  for any choice of $\eta > 0$, uniformly in $x_0$ and $x$.
  Furthermore, since the differential of a linear transformation is
  again the transformation itself,
      we have $\tilde{a}'_{x_0,t}(x) = P_{x_0,x}(a'_{x_0,t})$.
        Thus, by the above-mentioned continuity of $P_{x_0,x}(a)$, one easily sees
  that we can just as well shrink $\delta$ to get the following bound,
  uniform in $x_0$ and $x$:
  \begin{equation}\label{eq:5}
    L(\tilde{a}_{x_0,t}(x)) < L(a_{x_0,t}) + \epsilon < d_{x_0}(g_0(x_0),
    g_1(x_0)) + 2 \epsilon.
  \end{equation}
  Finally, we note that $d_x(g_0(x), P_{x_0,x}(g_0(x_0)))$,
  $d_x(g_1(x), P_{x_0,x}(g_1(x_0)))$, and $L(\tilde{a}_{x_0,t}(x))$
  are all continuous in $x$, since all of the quantities involved in
  their computation are continuous.

  For any $x_0 \in M$, $x \in B_\delta(x_0)$ and $\alpha \in \{0,1\}$,
  we let $\sigma^\alpha_{x_0,x,t}$, $t \in [0,1]$, be the geodesic in
  $\Matx$ connecting $g_\alpha(x)$ and $P_{x_0,x}(g_\alpha(x_0))$.  We
  assume that this geodesic is parametrized proportionally to arc
  length; note that in this case, $\sigma^\alpha_{x_0,x,t}$ varies
  continuously in $x$ on $B_\delta(x_0)$ for fixed $\alpha$, $x_0$,
  and $t$.  Referring to \eqref{eq:3}, we see that for given $x$ and
  $\alpha$, if $\eta > 0$ is small enough, then this geodesic indeed
  exists and is unique.  In fact, such a positive $\eta$ can be found
  independently of $x$ and $\alpha$ since $g_0(x)$ and $g_1(x)$ lie in
  the compact region $\cup_x \Matx^{\zeta,\tau} \subset S^2_+ T^* M$.
  We may shrink $\delta$ if necessary to insure that \eqref{eq:3} is
  satisfied for this $\eta$.

  Define metrics $\hat{g}_0$ and $\hat{g}_1$ by
  \begin{equation*}
    \hat{g}_0(x) :=
    \begin{cases}
      g_0(x) & \textnormal{if}\ x \not\in B_\delta(x_0), \\
      P_{x_0,x}(g_0(x_0)) & \textnormal{if}\ x \in B_\delta(x_0),
    \end{cases} \quad \textnormal{and} \quad
          \hat{g}_1(x) :=
    \begin{cases}
      g_0(x) & \textnormal{if}\ x \not\in B_\delta(x_0), \\
      P_{x_0,x}(g_1(x_0)) & \textnormal{if}\ x \in B_\delta(x_0).
    \end{cases}
  \end{equation*}
  (Note that both metrics equal $g_0$ on $M \setminus B_\delta(x_0)$.)
  We consider the paths
  \begin{equation*}
    g_{x_0,t}^0(x) :=
    \begin{cases}
      g_0(x) & \textnormal{if}\ x \not\in B_\delta(x_0), \\
      \sigma_{x_0,x,t}^0 & \textnormal{if}\ x \in B_\delta(x_0),
    \end{cases} \quad \textnormal{and} \quad
    g_{x_0,t}^1(x) :=
    \begin{cases}
      g_0(x) & \textnormal{if}\ x \not\in B_\delta(x_0), \\
      \sigma_{x_0,x,t}^1 & \textnormal{if}\ x \in B_\delta(x_0).
    \end{cases}
  \end{equation*}
  Then these are smooth paths in $\Mf$ that are continuous on
  $B_\delta(x_0)$.  The path $g_{x_0,t}^0$ connects $g_0$ and
  $\hat{g}_0$, while $g^1_{x_0,t}$ connects $\chibar(B_\delta(x_0))
  g_0 + \chi(B_\delta(x_0)) g_1$ and $\hat{g_1}$.  Furthermore, by
  shrinking $\delta$ to obtain $\eta < \epsilon$, we have by
  \eqref{eq:3} that for each $x_0 \in M$ and $x \in B_\delta(x_0)$, we
  have
  \begin{equation}\label{eq:6}
    L(g^\alpha_{x_0,t}(x)) = d_x(g_\alpha(x), P_{x_0,x}(g_\alpha(x_0))) < \epsilon
  \end{equation}
  for $\alpha = 0,1$.  We also have that on $B_\delta(x_0)$,
  $L(g^\alpha_{x_0,t}(x)) = d_x(g_\alpha(x), P_{x_0,x}(g_\alpha(x_0)))$
  is continuous in $x$, as noted after \eqref{eq:5}.
      
  We define a path $\hat{g}_{x_0,t}$ by
  \begin{equation}\label{eq:27}
    \hat{g}_{x_0,t}(x) :=
    \begin{cases}
      g_0(x) & \textnormal{if}\ x \not\in B_\delta(x_0), \\
      \tilde{a}_{x_0,t}(x) & \textnormal{if}\ x \in B_\delta(x_0).
    \end{cases}
  \end{equation}
  Then $\hat{g}_{x_0,t}$ is a path from $\hat{g}_0$ to $\hat{g}_1$.
  As noted above, both $\tilde{a}_{_0,t}(x)$ and
  $L(\tilde{a}_{x_0,t}(x))$ vary continuously with $x$ on
  $B_\delta(x_0)$.
                      
  Thus, by the following concatenation,
  \begin{equation}\label{eq:22}
    \bar{g}_{x_0,t} := g^0_{x_0,t} * \hat{g}_{x_0,t} * (g^1_{x_0,t})^{-1},
  \end{equation}
  we get a piecewise smooth path in $\Mf$ from $g_0$ to
  $\chibar(B_\delta(x_0)) g_0 + \chi(B_\delta (x_0)) g_1$ that in
  continuous on $B_\delta(x_0)$.  (Here, $(g^1_{x_0,t})^{-1}$ denotes
  the path $g^1_{x_0,t}$ run through in reverse.)  Let us assume that
  $\bar{g}_{x_0,t}$ is parametrized such that it runs through
  $g^0_{x_0,t}$ for $t \in [-\epsilon, 0]$, then $\hat{g}_{x_0,t}$ for
  $t \in [0,1]$, and finally $(g^1_{x_0,t})$ for $t \in [1, 1 +
  \epsilon]$.

  Denote by $g_{x_0,t}$ be the path obtained from $\bar{g}_{x_0,t}$ by
  pointwise reparametrizing each portion of the concatenation
  \eqref{eq:22} proportionally to arc length.
        Then by Lemma \ref{lem:6} and the statements following \eqref{eq:5},
  \eqref{eq:6}, and \eqref{eq:27}, $g_{x_0,t}$ is a piecewise smooth
  path in $\Mf$ that is continuous when restricted to $B_\delta(x_0)$,
  and by construction $g_{x_0,t}(x) = g_0(x)$ for all $t \in
  [-\epsilon,1+\epsilon]$ if $x \not\in B_\delta(x_0)$.  For $x \in
  B_\delta(x_0)$, the estimates (\ref{eq:5}) and (\ref{eq:6}) give
  \begin{equation*}
    \abs{g'_{x_0,t}(x)}_{g_{x_0,t}} <
    \begin{cases}
      1, & t \in [-\epsilon, 0) \cup [1, 1 + \epsilon], \\
      d_{x_0}(g_0(x_0), g_1(x_0)) + 2 \epsilon, & t \in [0,1).
    \end{cases}
  \end{equation*}
  
  Finally, since $g_0$ and $g_1$ are smooth metrics, the function $x
  \mapsto d_x(g_0(x), g_1(x))$ is continuous.  Therefore, we may
  assume that $\delta$ is small enough that $d_{x_0}(g_0(x_0),
  g_1(x_0)) < d_x(g_0(x), g_1(x)) + \epsilon$ for all $x \in
  B_\delta(x_0)$.  This and the above inequality show that $g_{x_0,t}$
  has all the desired properties.
\end{proof}

\begin{lemma}\label{lem:5}
  Let any $g_0, g_1 \in \M$ and $\epsilon > 0$ be given, and let
  $\delta = \delta(\epsilon, g_0, g_1) > 0$ be as in Lemma
  \ref{lem:3}.  Consider a finite collection of closed subsets $\{F_i
  \mid i = 1, \dots, N\}$ with the property that for each $i$, there
  exists $x_i \in F_i$ such that $F_i \subseteq B_{\delta'}(x_i)$ for
  some $0 < \delta' < \delta$, and such that $F_i \cap F_j =
  \emptyset$
    for all $i \neq j$.  We denote
  \begin{equation*}
    F := \bigcup_{i = 1, \dots, n} F_i.
  \end{equation*}

  Then there exists a path $\tilde{g}_t$, for $t \in [-\epsilon, 1 +
  \epsilon]$, from $g_0$ to $\tilde{g}_1 := \chibar(F) g_0 + \chi(F)
g_1$ satisfying
  \begin{equation}\label{eq:16}
    L(\tilde{g}_t)^2 < (1 + 2 \epsilon) \left[ \Omega_2(g_0, g_1)^2 + 6
        \epsilon \Omega_1(g_0, g_1) + 9 \epsilon^2 + 2
        \epsilon \right].
  \end{equation}
  Furthermore, $\tilde{g}_t$ satisfies the assumptions of Lemma
  \ref{lem:2}, and so also
  \begin{equation}\label{eq:17}
    d(g_0, \tilde{g}_1)^2 < (1 + 2 \epsilon) \left[ \Omega_2(g_0, g_1)^2 + 6
        \epsilon \Omega_1(g_0, g_1) + 9 \epsilon^2 + 2
        \epsilon \right].
  \end{equation}
\end{lemma}
\begin{proof}
  For each $i \in \N$, let $g_{i,t} := g_{x_i,t}$ be the path from
  $g_0$ to $\chibar(B_\delta(x_i)) g_0 + \chi(B_\delta (x_i)) g_1$
  guaranteed by Lemma \ref{lem:3}.  Then for each $x \in
  B_\delta(x_i)$, we have
  \begin{equation}\label{eq:10}
    \abs{g'_{i,t}(x)}_{g_{x_0,t}} <
    \begin{cases}
      1, & t \in [-\epsilon, 0) \cup [1, 1 + \epsilon], \\
      d_{x}(g_0(x), g_1(x)) + 3 \epsilon, & t \in [0,1).
    \end{cases}
  \end{equation}
          Additionally, $g_{i,t}(x)$ is constant in
  $t$ for $x \notin B_\delta(x_i)$.
                    
  Since the sets $F_i$ are pairwise disjoint and closed, we can
  find $\eta > 0$ such that the closed subsets
  \begin{equation*}
    B_\eta(F_i) = \{ x \in M \mid \dist_g(x, F_i) \leq \eta \}
  \end{equation*}
  are still pairwise disjoint.  (Here, $\dist_g$ denotes the distance
  function of the Riemannian metric $g$ on $M$.)  Since $F_i \subseteq B_{\delta'}(x_i)$
  for some $0 < \delta' < \delta$, we may also choose $\eta$ small
  enough that $B_\eta(F_i) \subseteq B_\delta(x_i)$ for all $i$.

  Now, for each $i$, we define a continuous function for $x
  \in B_\eta(F_i)$ by
  \begin{equation*}
    s_i(x, t) :=
    \left(
      \frac{\eta - \dist_g(x, F_i)}{\eta}
    \right) (t + \epsilon) - \epsilon,
  \end{equation*}
  so that $s_i(x,-\epsilon) \equiv -\epsilon$.  Furthermore, $s_i(x,t)
  = t$ for all $x \in F_i$ and $t \in [-\epsilon, 1 + \epsilon]$, and
  $s_i(x,t) = -\epsilon$ for all $t$ if $x \in \partial B_\eta(F_i)$.
  We define a smooth path in $\MC$ as follows:
  \begin{equation*}
    \bar{g}_t(x) :=
    \begin{cases}
      g_{i,t}(x), & x \in F_i, \\
      g_{i, s_i(x,t)}(x) & x \in B_\eta(F_i), \\
      g_0(x), & x \notin \cup_i B_\eta(F_i),
    \end{cases} \qquad \textnormal{for}\ t \in [-\epsilon, 1 + \epsilon].
  \end{equation*}

  With this definition, we can see that the path $\tilde{g}_t :=
  \chibar(F) g_0 + \chi(F) \bar{g}_t$ satisfies the assumptions of
  Lemma \ref{lem:2}, and hence $d(g_0, \tilde{g}_1) \leq
  L(\tilde{g}_t)$.  We claim that (\ref{eq:16}) and hence
  (\ref{eq:17}) hold as well.

  To see this, note that $\abs{\tilde{g}'_t(x)}_{\tilde{g}_t(x)} =
  \abs{g'_{i,t}(x)}_{g_{i,t}}$ for all $x \in F$.  Therefore,
  using \eqref{eq:10}, we can estimate
  \begin{equation}\label{eq:19}
    \begin{aligned}
      L(\tilde{g}_t)^2 &\leq (1 + 2\epsilon) E(\tilde{g}_t) = (1 +
      2\epsilon) \sum_{i=1}^N
      \integral{-\epsilon}{1 + \epsilon}{\integral{F_i}{}{\abs{\tilde{g}'_t(x)}_{\tilde{g}_t(x)}^2}{d
          \mu}}{dt} \\
      &< (1 + 2\epsilon) \left[ \sum_{i=1}^N
        \integral{0}{1}{\integral{F_i}{}{(d_{x}(g_0(x), g_1(x)) + 3
            \epsilon)^2}{d \mu}}{dt} + \sum_{i=1}^N
        \integral{[-\epsilon, 0) \cup [1, 1 +
          \epsilon]}{}{\integral{F_i}{}{}{d \mu}}{dt} \right] \\
      &= (1 + 2\epsilon) \left[ \sum_{i=1}^N
        \integral{F_i}{}{d_{x}(g_0(x), g_1(x))^2}{d \mu} \right.
        \\
        &\qquad \qquad \qquad \left. + 6 \epsilon \cdot \sum_{i=1}^N
        \integral{F_i}{}{d_{x}(g_0(x), g_1(x))}{d \mu} + (9 \epsilon^2
        + 2 \epsilon) \sum_{i=1}^N
        \integral{F_i}{}{}{d \mu} \right] \\
      &\leq (1 + 2 \epsilon) \left[ \Omega_2(g_0, \tilde{g}_1)^2 + 6
        \epsilon \Omega_1(g_0, \tilde{g}_1) + 9 \epsilon^2 + 2
        \epsilon \right].
    \end{aligned}
  \end{equation}
  The last line follows by the formulas for $\Omega_1$ and $\Omega_2$,
  as well as the fact that $\Vol(M, \mu) = 1$.

  Finally, we note that $\Omega_1(g_0, \tilde{g}_1) \leq \Omega_1(g_0,
  g_1)$ and $\Omega_2(g_0, \tilde{g}_1) \leq \Omega_2(g_0, g_1)$ since
  $\tilde{g}_1$ equals $g_1$ on $F$ and $g_0$ everywhere else.  Thus
  (\ref{eq:19}) in fact implies (\ref{eq:16}).
\end{proof}

We now have all the pieces necessary to prove the main result of this
section.

\begin{theorem}\label{thm:2}
  $d(g_0, g_1) = \Omega_2(g_0, g_1)$ for all $g_0, g_1 \in \M$.
\end{theorem}
\begin{proof}
  We have already shown in Theorem \ref{thm:1} that $d(g_0, g_1) \geq
  \Omega_2(g_0, g_1)$, so it only remains to show the reverse
  inequality.

  Let any $\epsilon > 0$ be given, and let $\delta = \delta(\epsilon,
  g_0, g_1)$ be the number guaranteed by Lemma \ref{lem:3}.
  
  Choose $0 < \delta' < \delta$ and a finite number of points $x_i \in
  M$, $i = 1, \dots, N$, such that $B_i :=
  \interior(B_{\delta'}(x_i))$ cover $M$.  (Here, $\interior$ denotes
  the interior of a set.)  For each $i$, we choose
    $0 < \delta_i < \delta'$
  such that
  \begin{equation}\label{eq:12}
    \max
    \left\{
      \Vol(B_i \setminus B_{\delta_i}(x_i), g_0), \Vol(B_i \setminus
      B_{\delta_i}(x_i), g_1) 
    \right\} < \frac{\epsilon}{2^N - 1}.
  \end{equation}
  
  We then let
    $F_1 := B_{\delta_1}(x_1)$.  For
  each $i = 2, \dots, N$, define
  \begin{equation*}
    F_i := B_{\delta_i}(x_i) \setminus \bigcup_{j < i} B_j.
  \end{equation*}
  We wish to see that the sets $F_i$ cover $M$ up to a set of measure
  $\epsilon$, intrinsically with respect to both $g_0$ and $g_1$. By
  (\ref{eq:12}), for $\alpha = 0,1$,
  \begin{equation*}
    \Vol(F_1, g_\alpha) \geq \Vol(B_1, g_\alpha) - \frac{\epsilon}{2^N - 1}.
  \end{equation*}
  To estimate $\Vol(F_1 \cup F_2, g_\alpha)$,
  note that
  \begin{equation*}
    F_1 \cup F_2 = B_{\delta_1}(x_1) \cup (B_{\delta_2}(x_2) \setminus
    B_1) = B_{\delta_1}(x_1) \cup (B_{\delta_2}(x_2) \setminus
    (B_1 \setminus B_{\delta_1}(x_1))).
  \end{equation*}
  The first set in the union on the right-hand side is completely
  contained in $B_1$, and the second set is completely contained in
  $B_2$.  Furthermore, they are disjoint.  Therefore, again using
  (\ref{eq:12}),
  \begin{equation}
    \begin{aligned}
      \Vol(F_1 \cup F_2, g_\alpha) &=
      \Vol(B_{\delta_1}(x_1), g_\alpha) + \Vol(B_{\delta_2}(x_2)
      \setminus (B_1 \setminus B_{\delta_1}(x_1))) \\
      &\geq
      \left(
        \Vol(B_1, g_\alpha) - \frac{\epsilon}{2^N - 1}
      \right) +             \left(
        \Vol(B_{\delta_2}(x_2)) - \Vol(B_1 \setminus B_{\delta_1}(x_1), g_\alpha)
      \right) \\
      &\geq \left(
        \Vol(B_1, g_\alpha) - \frac{\epsilon}{2^N - 1}
      \right) +
      \left(
        \Vol(B_{\delta_2}(x_2)) - \frac{\epsilon}{2^N - 1}
      \right) \\
      &\geq \left(
        \Vol(B_1, g_\alpha) - \frac{\epsilon}{2^N - 1}
      \right) +
      \left(
        \Vol(B_2, g_\alpha) - 2 \cdot \frac{\epsilon}{2^N - 1}
      \right) \\
      &\geq \Vol(B_1 \cup B_2, g_\alpha) - (1 + 2) \frac{\epsilon}{2^N - 1}.
    \end{aligned}
  \end{equation}
  If we continue in this way, we find that for $F := \cup_i F_i$,
  \begin{equation}\label{eq:14}
    \Vol(F, g_\alpha) \geq \Vol(M, g_\alpha) -
    \left(
      \sum_{j = 0}^{N - 1} 2^j
    \right) \frac{\epsilon}{2^N - 1} = \Vol(M, g_\alpha) - \epsilon,
  \end{equation}
  where we recall that $\cup_i B_i = M$.

  Now, note that as defined, the sets $F_i$ satisfy the assumptions of
  Lemma \ref{lem:5}.  Let $\tilde{g}_t$ and $\tilde{g}_1$ be as in
  the lemma.  Then we have that
  \begin{equation}\label{eq:13}
    d(g_0, \tilde{g}_1)^2 < (1 + 2 \epsilon) \left[ \Omega_2(g_0, g_1)^2 + 6
        \epsilon \Omega_1(g_0, g_1) + 9 \epsilon^2 + 2
        \epsilon \right].
  \end{equation}
  On the other hand, $\tilde{g}_1$ and $g_1$ differ only on $M
  \setminus F$, where $\tilde{g}_1 = g_0$.  Thus, by (\ref{eq:14}) and
  Proposition \ref{prop:3}, we have
  \begin{equation}
    \label{eq:15}
    d(\tilde{g}_1, g_1) \leq C(n)
    \left(
      \sqrt{\Vol(M \setminus F, g_0)} + \sqrt{\Vol(M \setminus F, g_1)}
    \right) < 2 C(n) \sqrt{\epsilon}.
  \end{equation}
  Applying the triangle inequality to (\ref{eq:13}) and (\ref{eq:15}),
  we obtain
  \begin{equation*}
    d(g_0, g_1) < \sqrt{(1 + 2 \epsilon) \left[ \Omega_2(g_0, \tilde{g}_1)^2 + 6
        \epsilon \Omega_1(g_0, \tilde{g}_1) + 9 \epsilon^2 + 2
        \epsilon \right]} + 2 C(n) \sqrt{\epsilon}.
  \end{equation*}
  Sending $\epsilon \rightarrow 0$ then gives the desired result,
  $d(g_0, g_1) \leq \Omega_2(g_0, g_1)$.
\end{proof}

Since the completion of $(\M, d)$ is already known, the previous
theorem implies that the $L^2$ completion of $\M$ (in the sense
discussed following \eqref{eq:25}) is given by $\Mfhat$.

\section{Minimal paths in $\overline{\M}$ and $\overline{\Matx}$}
\label{sec:geod-exist-uniq}

In this section, we use the formula $d = \Omega_2$, together with an
analysis of the geometry of the fiber spaces $(\Matx, \langle \cdot,
\cdot \rangle)$, to obtain results about geodesics in $\overline{\M}$.
(We will use the terms \emph{geodesic} and \emph{minimal path}
interchangeably to mean a path whose length minimizes the distance
globally, i.e., between any two of its points.  If we are referring to
geodesics in the sense of Riemannian geometry, we will refer to
\emph{Riemannian geodesics}.)  Of course, as the completion of the
path metric space $(\M, d)$, $\overline{\M}$ is itself a complete path
metric space \cite[3.6(3)]{Bridson:1999tu}.  (By \emph{path metric
  space}, we mean that the distance between points is equal to the
infimum of the lengths of paths between those points.  Some authors
refer to this as an inner or intrinsic metric space.)  However, since
$\overline{\M}$ is not locally compact, completeness is no guarantee
that minimal paths exists between arbitrary points---even in the
Riemannian case, this does not hold, as an example of McAlpin shows
\cite[Sect.~I.E]{mcalpin65:_infin_dimen_manif_and_morse_theor} (see
also \cite[Sect.~VIII.6]{lang95:_differ_and_rieman_manif}).

Nevertheless, we will show that a unique geodesic exists between any
two points, and we can give an explicit and easily computable formula
for this geodesic.  Our analysis of the geometry of $\Matx$ builds
upon the foundation set up by Freed--Groisser
\cite{freed89:_basic_geomet_of_manif_of} and Gil-Medrano--Michor
\cite{gil-medrano91:_rieman_manif_of_all_rieman_metric}.

To begin this section, we give some estimates for the distance
function $d_x$ of $(\Matx, \langle \cdot, \cdot \rangle)$, and
determine the completion $\overline{(\Matx, d_x)}$.  Following that,
we summarize the work of Freed--Groisser and Gil-Medrano--Michor
mentioned above.  We use that work to help determine minimal paths
between points of $\overline{(\Matx, d_x)}$, and show that these exist
and are unique.  Following that, especially for use in Section
\ref{sec:cat0-property}, we show that these minimal paths in
$\overline{\Matx}$ vary continuously with their endpoints.  Finally,
we use the preceding results and Theorem \ref{thm:2} to give explicit
formulas for the distance function and minimal paths of
$\overline{\M}$.

\subsection{The metric $d_x$ on $\Matx$}
\label{sec:metric-d_x-matx}

We begin our investigation of $d_x$ by giving estimates of it from
below and above, which allows us to determine the completion of
$(\Matx, d_x)$.

Given a tensor $a \in \satx$, we have as before $A = g(x)^{-1} a$.
For positive semi-definite $a$, we will denote by $\sqrt{A}$ the
square root of the determinant of $A$.  Similarly, $\sqrt[4]{A}$
simply denotes $\sqrt[4]{\det A}$.  Note that these quantities are
coordinate-independent since $A$ is an endomorphism of $T_x M$.

Our first result bounds $d_x$ from below based on the determinants of
two given elements, and is the pointwise analog of \cite[Lemma
12]{clarke:_metric_geomet_of_manif_of_rieman_metric}.  It will come in
useful when showing that given paths are minimal in $\Matx$.

\begin{lemma}\label{lem:8}
  Let $a_0, a_1 \in \Matx$.  Then
  \begin{equation*}
    d_x(a_0, a_1) \geq \frac{4}{\sqrt{n}} \abs{\sqrt[4]{A_1} -
      \sqrt[4]{A_0}}.
  \end{equation*}
  In particular the functions $a \mapsto \det A$ and $a \mapsto
  \sqrt[4]{A}$ are continuous and Lipschitz continuous, respectively,
  on $(\Matx, d_x)$.
\end{lemma}
\begin{proof}
  The proof is essentially the same as \cite[Lemma
  12]{clarke:_metric_geomet_of_manif_of_rieman_metric}, but for
  completeness we include it here.
  
  First, let $a \in \Matx$, and suppose that $b \in T_a \Matx \cong
  \satx$.  Let $b_1$ be the pure-trace part of $b$ ($b_1 = \frac{1}{n}
  \tr_a b$) and $b_0$ be the trace-free part ($b_0 = b - b_1$).  It is
  easy to see that $\tr_a(b_0 b_1) = 0$, and therefore
  \begin{equation*}
    \tr_a(b^2) = \tr_a(b_0^2) + \tr_a(b_1^2) = \tr_a(b_0^2) +
    \frac{1}{n} (\tr_a b)^2. 
  \end{equation*}
  Since $\tr_a(b_0^2) \geq 0$, we can conclude that $(\tr_a b)^2 \leq
  n \tr_a(b^2)$.
  
  Let $a_t$, $t \in [0,1]$, be any path connecting $a_0$ and $a_1$.
  We can estimate
  \begin{align*}
    \partial_t \sqrt[4]{A_t} &= \frac{1}{4} (\det A_t)^{-3/4}
    \tr_{a_t} (a'_t) (\det A_t) = \frac{1}{4} \left( \tr_{a_t}(a'_t)^2 \sqrt{A_t}
    \right)^{1/2} \\
    &\leq \frac{1}{4} \left(
      n \tr_{a_t}((a'_t)^2) \sqrt{A_t}
    \right)^{1/2} = \frac{\sqrt{n}}{4} \norm{a'_t}_{a_t},
  \end{align*}
  where we have used the inequality of the last paragraph.
  Integrating this last estimate gives
  \begin{equation*}
    \sqrt[4]{A_1} - \sqrt[4]{A_0} = \integral{0}{1}{\partial_t
      \sqrt[4]{A_t}}{dt} \leq \frac{\sqrt{n}}{4}
    \integral{0}{1}{\norm{a'_t}_{a_t}}{dt} = \frac{\sqrt{n}}{4} L(a_t).
  \end{equation*}
  Since this holds for any path $a_t$ between $a_0$ and $a_1$, and we
  can just as easily exchange $a_0$ and $a_1$, the statement of the
  lemma is proved.
          \end{proof}

Now, let $a_1 \in \Matx$ and consider the path $a_t = t a_1$, $t \in
(0, 1]$.  It is straightforward to compute that $L(a_t) =
\frac{4}{\sqrt{n}} \sqrt[4]{A_1}$, and therefore we can think of the
zero tensor (denoted simply by $0$) as representing a point in the
completion of $(\Matx, d_x)$.  (We will make this more precise in
Proposition \ref{prop:1} below.)  The metric $d_x$ naturally extends
to a metric on the completion (which we again denote by $d_x$).  By
Lemma \ref{lem:8}, $a_t$ is minimal, so we have $d_x(a_1, 0) =
\frac{4}{\sqrt{n}} \sqrt[4]{A_1}$ for any $a_1 \in \Matx$.  Thus, by
an application of the triangle inequality, we have the following
``converse'' of Lemma \ref{lem:8}.  It is a pointwise analog of
Proposition \ref{prop:3}.

\begin{lemma}\label{lem:9}
  Let $a_0, a_1 \in \Matx$.  Then
  \begin{equation*}
    d_x(a_0, a_1) \leq \frac{4}{\sqrt{n}} \left( \sqrt[4]{A_0} +
      \sqrt[4]{A_1} \right). 
  \end{equation*}
\end{lemma}

We now wish to determine the completion of $(\Matx, d_x)$, which we
will do by comparison with another metric.  Consider the Riemannian
metric $\langle \cdot, \cdot \rangle^0$ on $\Matx$ given by
\begin{equation*}
  \langle b, c \rangle^0_{a} = \tr_a(bc).
\end{equation*}
This metric turns $\Matx$ into a complete Riemannian
manifold---indeed, into a symmetric space (see
\cite[Thm.~8.9]{ebin70:_manif_of_rieman_metric} or
\cite[Prop.~4.9]{clarke:_compl_of_manif_of_rieman_metric}).  Since the
scalar product $\langle \cdot, \cdot \rangle_a$ differs from $\langle
\cdot, \cdot \rangle^0_a$ only by the factor $\sqrt{A}$, one
reasonably suspects that the only points that could be missing from
the completion of $(\Matx, \langle \cdot, \cdot \rangle)$ are those
with determinant zero.  The next proposition confirms this hunch and
makes it rigorous.

\begin{proposition}\label{prop:1}
  The completion of $(\Matx, \langle \cdot, \cdot
  \rangle)$ can be identified with
  \begin{equation*}
    \overline{\Matx} \cong \grpquot{\cl(\Matx)}{\partial \Matx},
  \end{equation*}
  where $\cl(\Matx)$ denotes the topological closure of $\Matx$ as a
  subspace of $\satx$, and $\partial \Matx$ denotes the boundary in
  $\satx$.

  The topology is given by the following.  Given a sequence $\{a_k\}
  \subset \overline{\Matx}$, it converges to $a_0 \in \Matx$ if and
  only if it does so in the manifold topology, and it converges to
  $[0] \in \Matx$ (the equivalence class of the zero tensor) if and
  only if $\det A_k \rightarrow 0$.  In fact, $d_x(a, [0]) =
  \frac{4}{\sqrt{n}} \sqrt[4]{A}$ for any $a \in \overline{\Matx}$.
\end{proposition}
\begin{proof}
  By the standard construction of the completion of a Riemannian
  manifold, we must consider all piecewise differentiable paths of the
  form $a_t$, $t \in[0,1)$, in $\Matx$ that have finite length with
  respect to $\langle \cdot, \cdot \rangle$ and show two facts.
  First, either $\lim_{t \rightarrow 1} a_t \in \Matx$ (in the
  topology of $\satx$) or $\lim_{t \to 1} \det A_t = 0$.  Second, if
  $\lim_{t \to 1} \det A_t = 0$ and $\tilde{a}_t$, $t \in [0,1)$, is
  another path in $\Matx$ satisfying $\lim_{t \to 1} \det \tilde{A}_t
  = 0$, then $a_t$ and $\tilde{a}_t$ are equivalent in the sense that
  $\lim_{t \to 1} d_x(a_t, \tilde{a}_t) = 0$.  (From these facts, the
  statements about the topology on $\overline{\Matx}$ follow
  immediately.)

  The second fact, however, is immediate from Lemma \ref{lem:9}.  So
  to prove the first fact, suppose we do not have $\lim_{t \to 1} \det
  A_t = 0$.  By Lemma \ref{lem:8}, one can easily see that $\det A_t$
  must nevertheless converge to some limit $\eta > 0$.  Furthermore,
  Lemma \ref{lem:8} implies that there exists $\epsilon > 0$ such that
  $\eta / 2 < \det A_t < 3 \eta / 2$ for all $t \in [1-\epsilon, 1)$.
  But since $\langle \cdot, \cdot \rangle$ is equivalent to $\langle
  \cdot, \cdot \rangle^0$ on the subset $\{ a \in \Matx \mid \eta / 2
  < \det A < 3 \eta / 2 \}$, the completeness of $\langle \cdot, \cdot
  \rangle^0$ implies that $\lim_{t \rightarrow 1} a_t \in \Matx$, as
  desired.

  The formula for $d_x(a, [0])$ holds by the discussion following
  Lemma \ref{lem:8}.
\end{proof}

As in the case of $\overline{\M}$, the completion $\overline{\Matx}$
together with the metric induced from $d_x$ (which we will again
denote by $d_x$) is a path metric space.  Furthermore, given
Proposition \ref{prop:1}, we see that Lemmas \ref{lem:8} and
\ref{lem:9} continue to hold if $a_0, a_1 \in \overline{\Matx}$, so
from now on we will assume the lemmas are stated as such.

\subsection{Riemannian geodesics in $\Matx$}
\label{sec:riem-geod-matx}

In this subsection, we recall some results on geodesics and the
exponential mappings for $(\Matx, \langle \cdot, \cdot \rangle)$.

As explained in \cite[Appendix]{freed89:_basic_geomet_of_manif_of},
the formulas for geodesics on $\Matx$ follow directly from those on
$\M$ determined by Freed--Groisser
\cite[Thm.~2.3]{freed89:_basic_geomet_of_manif_of} and
Gil-Medrano--Michor
\cite[Thm.~3.2]{gil-medrano91:_rieman_manif_of_all_rieman_metric}.
Namely, we have that $g_t$ is a geodesic in $(\M, (\cdot, \cdot))$ if
and only if $g_t(x)$ is a geodesic in $(\Matx, \langle \cdot, \cdot
\rangle)$.  Therefore, in the following we will quote formulas that
Freed--Groisser and Gil-Medrano--Michor formulated for $\M$,
translated into the result for $\Matx$.

For the remainder of the paper, we denote by $b_T := b - \frac{1}{n}
(\tr_{a_0} b) a_0$ the traceless part of any $b \in T_{a_0} \Matx
\cong \satx$.  Furthermore, in all that follows, $\exp$ without a
subscript denotes the usual exponential of a matrix or linear
transformation, while $\exp_{a_0}$ denotes the Riemannian exponential
mapping of $a_0 \in \Matx$.

\begin{theorem}\label{thm:4}
  Let $a_0 \in \Matx$ and $b \in T_{a_0} \Matx \cong \satx$.  Define 
  \begin{equation*} q(t) := 1 + \frac{t}{4} \tr_{a_0}(b), \quad r(t)
:= \frac{t}{4} \sqrt{n \tr_{a_0}(b_T^2)}.
  \end{equation*}
  Then the geodesic starting at $a_0$ with initial tangent $a'_0 = b$
  is given by
  \begin{equation}\label{eq:28}
    a_t =
    \begin{cases}
      \left( q(t)^2 + r(t)^2 \right)^{\frac{2}{n}} a_0 \exp \left(
        \frac{4}{\sqrt{n \tr(b_T^2)}} \arctan
        \left( \frac{r(t)}{q(t)} \right) b_T \right) &
      \textnormal{if}\ b_T
      \neq 0, \\
      q(t)^{4/n} a_0 & \textnormal{if}\ b_T = 0.
    \end{cases}
  \end{equation}
  In particular, the change in the volume element $\sqrt{A_t}$ is given by
  \begin{equation}\label{eq:21}
    \sqrt{A_t} = (q(t)^2 + r(t)^2) \sqrt{A_0}.
  \end{equation}
  
  For precision, we specify the range of $\arctan$ in the above.  At a
  point where $\tr_{a_0} b \geq 0$, it assumes values in
  $(-\frac{\pi}{2}, \frac{\pi}{2})$.  At a point where $\tr_{a_0} b <
  0$, $\arctan(r(t) / q(t))$ assumes values as follows, with $t_0 := -
  \frac{4}{\tr_{a_0} b}$:
  \begin{enumerate}
  \item in $[0, \frac{\pi}{2})$ if $0 \leq t < t_0$,
  \item in $(\frac{\pi}{2}, \pi)$ if $t_0 < t < \infty$,
  \end{enumerate}
  and we set $\arctan(r(t) / q(t)) = \frac{\pi}{2}$ if $t = t_0$.

  Finally, the geodesic is defined on the following domain.  If $b_T =
  0$ and $\tr_{a_0} b < 0$, then the geodesic is defined for $t \in
  [0, t_0)$.  Otherwise, the geodesic is defined on
  $[0, \infty)$.
\end{theorem}

We note here that if, in the above, $b_T = 0$, i.e., the initial
tangent vector of the geodesic is pure-trace, then the geodesic is a
certain parametrization of a straight ray for as long as it is
defined.

Gil-Medrano--Michor also performed a detailed analysis of the
exponential mapping of $\M$.  We quote here a portion of their
results, translated into the pointwise result for $\Matx$. 

\begin{theorem}[{\cite[\S 3.3, Thm.~3.4]{gil-medrano91:_rieman_manif_of_all_rieman_metric}}]\label{thm:7}
  Let $a_0 \in \Matx$ and $U := \satx \setminus (-\infty, -4 / n ]\,
  a_0$.  Then $U$ is the maximal domain of definition of $\exp_{a_0}$,
  and $\exp_{a_0}$ is a diffeomorphism between $U$ and
  \begin{equation*}
    V := \exp_{a_0}(U) =
    \left\{
      a_0 \exp (a_0^{-1} b) \midmid \tr_{a_0}(b_T^2) < \frac{(4 \pi)^2}{n}
    \right\}.
  \end{equation*}

  The inverse of $\exp_{a_0}$ is given by the following.  For
      $b \in \satx$, define
  \begin{equation}\label{eq:29}
    \psi(b) := 
    \begin{cases}
      \frac{4}{n}
      \left(
        \exp
        \left(
          \frac{\tr_{a_0} b}{4}
        \right) \cos
        \left(
          \frac{\sqrt{n \tr_{a_0}(b_T^2)}}{4}
        \right)
        -1
      \right) a_0 & \\
      \qquad + \frac{4}{\sqrt{n \tr_{a_0}(b_T^2)}} \exp
      \left(
        \frac{\tr_{a_0} b}{4}
      \right) \sin
      \left(
        \frac{\sqrt{n \tr_{a_0}(b_T^2)}}{4}
      \right) b_T
      & \textnormal{if}\ b_T \neq 0 \\
      \frac{4}{n}
      \left(
        \exp
        \left(
          \frac{\tr_{a_0} b}{4}
        \right) - 1
      \right) a_0 & \textnormal{if}\ b_T = 0.
    \end{cases}
  \end{equation}
  Now, if $a_1 \in V$, write (uniquely) $a_1 = a_0 \exp(a_0^{-1} b)$
  for some $b \in \satx$.  Then $\exp_a^{-1}(a_1) = \psi(b)$.
\end{theorem}

In what follows, we will also require some facts about the exponential
mapping of certain submanifolds of $\Matx$, as defined below.

\begin{definition}\label{dfn:5}
  Let $a \in \Matx$.  We define
  \begin{equation*}
      \M_{x,\sqrt{A}} := \{ b \in \Matx \mid \sqrt{B} = \sqrt{A} \}.
  \end{equation*}
\end{definition}

Note that since the derivative of the map $b \mapsto \sqrt{B}$ is $c
\mapsto \frac{1}{2} \tr_b(c) \sqrt{B}$
\cite[Prop.~1.186]{Besse:2008vf} (cf.~the errata in the previous
reference), we have
\begin{equation*}
  T_b \M_{x, \sqrt{A}} = \{ c \in \satx \mid \tr_b c = 0 \}
\end{equation*}
for all $b \in \M_{x,\sqrt{A}}$.  Also, since for any $b \in \Matx$,
there exists a number $\lambda > 0$ such that $\sqrt{\lambda B} =
\sqrt{A}$, we have $\Matx \cong \R_+ \times \M_{x,\sqrt{A}}$.  This
decomposition is orthogonal with respect to $\langle \cdot, \cdot
\rangle$, since if $h = \lambda b$, $\lambda \in \R$, is tangent to
$\R_+ \cdot b$ and $k \in T_b \M_{x,\sqrt{A}}$, then $\tr_b(h k) =
\lambda \tr_b k = 0$.  We call vectors tangent to $\R_+ \cdot b$
\emph{pure-trace}, and those tangent to $T_b \M_{x,\sqrt{A}}$
\emph{traceless}.

By \cite[Thm.~8.9]{ebin70:_manif_of_rieman_metric} (see also
\cite[Prop.~1.13]{freed89:_basic_geomet_of_manif_of}),
$\M_{x,\sqrt{A}}$ is (non-canonically) isometric to the symmetric
space $SL(n) / SO(n)$.  Furthermore, we have the following formula for
the exponential mapping $\exp^0_b$ of $\M_{x, \sqrt{A}}$ (where we
have again translated the result to a pointwise one):

\begin{theorem}\label{thm:8}
  Let $b \in \M_{x, \sqrt{A}}$ and $c \in T_b \M_{x, \sqrt{A}}$.  Then
  \begin{equation*}
    \exp^0_b(c) = b \exp(b^{-1} c)
  \end{equation*}
  and $\exp^0_b$ is a diffeomorphism from $T_b \M_{x,\sqrt{A}}$ to
  $\M_{x,\sqrt{A}}$.
\end{theorem}

\subsection{Existence and uniqueness of geodesics in
  $\overline{\Matx}$}
\label{sec:exist-uniq-geod}

We now turn to the proof of existence and uniqueness of geodesics in
the completion $\overline{\Matx}$.  These geodesics will turn out to
have a relatively simple form: they are either Riemannian geodesics,
or they are the concatenation of the geodesic from the initial point
to the singular point $[0]$ of $\overline{\Matx}$, followed by the
geodesic from $[0]$ to the terminal point.  As we saw in \S
\ref{sec:metric-d_x-matx}, such a geodesic is consequently a
concatenation of straight segments.

On the other hand, note that $\overline{\Matx}$ is not locally
compact, as by Proposition \ref{prop:1}, the closed ball around $[0]$ of
radius $r > 0$ is the noncompact set
\begin{equation*}
  \overline{B}(r, [0]) = \{ [0] \} \cup \left\{ a \in \Matx \midmid \frac{4}{\sqrt{n}}
    \sqrt[4]{A} \leq r \right\}.
\end{equation*}
For this reason, existence of geodesics in $\overline{\Matx}$ does not
follow from general theory, e.g., the Hopf--Rinow--Cohn--Vossen
Theorem \cite[Prop.~1.3.7]{Bridson:1999tu}.  Therefore, it falls upon
us to prove this existence directly.

We begin this subsection with a general result on
\emph{non}-mimimality of paths, which we will then use to prove that
Riemannian geodesics in $\Matx$, which are unique by Theorem
\ref{thm:7}, minimize as long as they exist.  Following that, we
analyze the boundary of the image of exponential mapping of $\Matx$,
which will allow us to show the full existence and uniqueness result.

\subsubsection{Non-minimal paths}
\label{sec:non-minimal-paths}

At this point, we require a fundamental result in Riemannian geometry,
adapted to a low-regularity situation.  We begin with two technical
lemmas.

\begin{lemma}\label{lem:17}
  Let $[a,b]$ be an interval in $\R$, and let $r: [a,b] \rightarrow
  \R$ be $C^1$.  For $t \in [a,b]$, define
  \begin{equation*}
    \hat{r}(t) := \min \left( \max_{s \in [0, t]} r(s), 1 \right).
  \end{equation*}
  Then $\hat{r} : [a,b] \rightarrow \R$ is absolutely continuous.  In
  particular, $\hat{r}$ is a.e.-differentiable (with respect to
  Lebesgue measure), $\hat{r}'$ is integrable, and
  \begin{equation*}
    \hat{r}(b) - \hat{r}(a) = \integral{a}{b}{\hat{r}'(t)}{dt}.
  \end{equation*}
\end{lemma}
\begin{proof}
  We note that $\hat{r}$ is a continuous, monotone function.  Thus, by
  \cite[Thm.~7.18]{Rudin:1987tl}, if we can show that $\hat{r}$ maps
  sets of measure zero to sets of measure zero, then $\hat{r}$ is
  absolutely continuous, and
    the other properties follow.

  Now, for any $t \in [a,b]$, either $\hat{r}(t) = r(t)$, $\hat{r}(t)
  = 1$, or there exists a critical point $t_0$ of $r$ with $t_0 < t$
  and $\hat{r}(t) = r(t_0)$.  Let $\textnormal{Crit}(r)$ denote the
  set of critical points of $r$; by Sard's Theorem,
  $r(\textnormal{Crit}(r))$ has measure zero.  But if $N \subseteq
  [a,b]$ is any Lebesgue null set, then $r(N)$ is also a null set, and
  by the above discussion, $\hat{r}(N) \subseteq r(N) \cup
  r(\textnormal{Crit}(N)) \cup \{1\}$, so $\hat{r}(N)$ has measure
  zero, as desired.
\end{proof}

\begin{lemma}\label{lem:18}
  Let $(N, \gamma)$ be a finite-dimensional Riemannian manifold, and
  let $f : [0,1] \rightarrow \R$ be an absolutely continuous function.
  Let $p \in N$ and $v \in T_p N$ be such that $f(t) v$ lies in the
  domain of definition of $\exp_p$ for all $t$.  Then the radial path
  $\sigma(t) := \exp_p(f(t) v)$ is an absolutely continuous curve in
  $N$.  Furthermore, $\sigma$ is a.e.-differentiable,
  $\abs{\sigma'(t)}_\gamma = \abs{f'(t)} \abs{v}_\gamma$ for a.e.~$t$,
  and
  \begin{equation*}
    L(\sigma) = \integral{0}{1}{\abs{\sigma'(t)}_\gamma}{dt} =
    \abs{v}_\gamma \integral{0}{1}{\abs{f'(t)}}{dt}.
  \end{equation*}
\end{lemma}
\begin{proof}
  Since $f(t) v$ lies within the domain of definition of $\exp_p$ for
  all $t$, we may use the exponential mapping to apply classical
  results about absolutely continuous functions.
            
  In particular, since $\exp_p$ is smooth and $f$ is absolutely
  continuous, $\sigma$ is absolutely continuous, and hence
  a.e.-differentiable.  By the Gauss Lemma
  \cite[Lem.~1.9.1]{klingenberg95:_rieman_geomet}, the differential of
  $\exp_p$ maps $f'(t) v$ isometrically into $\sigma'(t)$ (wherever
  these exist), so $\abs{\sigma'(t)}_\gamma = \abs{f'(t)}
  \abs{v}_\gamma$.  Finally, the formula for the length of $\sigma$
  follows from
    \cite[p.~159]{Rudin:1987tl}.
  \end{proof}

The following theorem is the result we need---it is a modification and
extension of \cite[Thm.~1.9.2]{klingenberg95:_rieman_geomet}.

\begin{theorem}\label{thm:6}
  Let $(N, \gamma)$ be a finite-dimensional Riemannian manifold, $p
  \in N$, $v \in T_p N$ nonzero.  Let $\tilde{\rho}(t)$, $t \in [0,
  1]$, be a differentiable curve in $T_p M$ from $0$ to $v$ that lies
  within the domain of definition of $\exp_p$ for all $t$.  Assume
  $\tilde{\rho}(t) \neq 0$ for $t > 0$ and write $\tilde{\rho}(t)$ in
  polar coordinates,
  \begin{equation*}
    \tilde{\rho}(t) = r(t) w(t), \qquad w(t) :=
    \tilde{\rho}(t) / \abs{\tilde{\rho}(t)}_\gamma.
  \end{equation*}
  Define
  \begin{equation*}
    \hat{r}(t) := \min \left( \max_{s \in [0, t]} r(s), 1 \right),
  \end{equation*}
  as well as $\tilde{\sigma}(t) := \hat{r}(t) v / \abs{v}_\gamma$ and
  $\tilde{\tau}(t) := t v$ for $t \in [0, 1]$.  Finally, define
  $\rho(t) := \exp_p(\tilde{\rho}(t))$, $\sigma(t) :=
  \exp_p(\tilde{\sigma}(t))$, and $\tau(t) :=
  \exp_p(\tilde{\tau}(t))$.

  Then $\rho$, $\sigma$ and $\tau$ are paths in $N$ from $p$ to $q :=
  \exp_p(v)$.  The path $\sigma$ is differentiable for
  a.e.~$t \in [0,1]$, and $L(\sigma) =
  \integral{0}{1}{\abs{\sigma'(t)}_\gamma}{dt}$.  For every $t$ where
  this is well defined, we have $\abs{\sigma'(t)}_\gamma \leq
  \abs{\rho'(t)}_\gamma$.  Furthermore, $L_\gamma(\tau) =
  L_\gamma(\sigma) \leq L_\gamma(\rho)$.

  Suppose the differential $D_{s \tilde{\rho}(t)} \exp_p$ of the
  exponential mapping is of maximal rank for all $s, t \in [0, 1]$.
  Then if and only if $\tilde{\rho}$ is a reparametrization of
  $\tilde{\tau}$, the following equalities hold: $L_\gamma(\rho) =
  L_\gamma(\tau) = L_\gamma(\sigma)$ and $\abs{\sigma'(t)}_\gamma =
  \abs{\rho'(t)}_\gamma$ for a.e.~$t$.  In particular, if $\exp_p$ is
  a diffeomorphism from $U \subseteq T_p N$ to $V \subseteq N$, then
  $\tau$ is of minimal length among all paths in $V$ from $p$ to $q$,
  and it is unique (up to reparametrization) with respect to this
  property.
\end{theorem}
\begin{proof}
  We first note that by Lemma \ref{lem:17}, $\hat{r}$ is absolutely
  continuous, so Lemma \ref{lem:18} applies to give that $\sigma$ is
  a.e.-differentiable, $\abs{\sigma'(t)}_\gamma = \abs{\hat{r}'(t)}$
  wherever these are defined, and $L(\sigma) =
  \integral{0}{1}{\abs{\sigma'(t)}_\gamma}{dt}$.
  
  Let $a := \abs{v}_\gamma$.  For small $\epsilon > 0$, we define a
  map
  \begin{equation*}
    F : [0,1] \times (\epsilon, 1] \rightarrow N; \qquad (s, t)
    \mapsto \exp_p(a s w(t)).
  \end{equation*}
  Then we have $\rho(t) = F(r(t) / a, t)$ and $\rho'(t) =
  \frac{r'(t)}{a} \frac{\partial F}{\partial s} \left( \frac{r(t)}{a},
    t \right) + \frac{\partial F}{\partial t} \left( \frac{r(t)}{a}, t
  \right)$.  Furthermore, by the Gauss Lemma
  \cite[Lem.~1.9.1]{klingenberg95:_rieman_geomet}, $u_1 := \partial
  F(s,t) / \partial s$ is orthogonal to $u_2 := \partial F(s,t)
  / \partial t$, and $\abs{\partial F(s,t) / \partial s}_\gamma =
  \abs{a w(t)}_\gamma = a$.  Therefore, for a.e.~$t$,
                            \begin{equation}\label{eq:37}
    \abs{\rho'(t)}_\gamma^2 = \abs{r'(t)}^2 + \abs{u_2 \left(
        \frac{r(t)}{a}, t \right)}_\gamma^2 \geq \abs{r'(t)}^2 \geq
    \abs{\hat{r}'(t)}^2 = \abs{\sigma'(t)}_\gamma^2.
  \end{equation}
          This shows that $L_\gamma(\sigma) \leq L_\gamma(\rho)$.
  Furthermore, we note that $\sigma$ is just a reparametrization of
  $\tau$, since $\sigma(t) = \tau(\hat{r}(t) / \abs{v}_\gamma)$, and
  $\hat{r}$ is a continuous, monotone function.  Therefore
  $L_\gamma(\tau) = L_\gamma(\sigma)$ follows
  (cf.~\cite[Prop.~I.1.20]{Bridson:1999tu}).

  If equality holds in \eqref{eq:37}, then $r'(t) = \hat{r}'(t) \geq
  0$.  Additionally, we must have
  \begin{equation*}
    0 = \left. \frac{\partial F}{\partial t} (s, t) \right|_{(s,t) =
      (r(t)/r, t)} = \left. D_{r s w(t)} \exp_p(r s w'(t)) \right|_{(s,t) =
      (r(t)/r, t)}.
  \end{equation*}
  Since $r(t) \neq 0$ for $t > 0$, if $D_{s \tilde{\rho}(t)} \exp_p$
  is of maximal rank for all $s, t \in [0, 1]$, then $w'(t) = 0$.
  Thus, if equality holds in \eqref{eq:37} for a.e.~$t$, then $w(t)
  \equiv w(1) = v / \abs{v}_\gamma$.  This shows that $\rho$ is a
  reparametrization of $\tau$.

  Finally, using \eqref{eq:37}, we have
  \begin{equation}\label{eq:38}
    L_\gamma(\rho) = \integral{0}{1}{\abs{\rho'(t)}_\gamma}{dt} \geq
    \integral{0}{1}{\abs{r'(t)}_\gamma}{dt} \geq
    \integral{0}{1}{r'(t)}{dt} = r(1) = L_\gamma(\tau).
  \end{equation}
  So if $L_\gamma(\tau) = L_\gamma(\rho)$, then by \eqref{eq:37} and
  \eqref{eq:38}, $r'(t) \geq 0$ for all $t$, and similarly to the
  previous paragraph, if $D_{s \tilde{\rho}(t)} \exp_p$ is of maximal
  rank for all $s, t \in [0, 1]$ we have $w'(t) \equiv 0$, so again
  $\rho$ is a reparametrization of $\tau$.
\end{proof}

The above general theorem has the following consequence in our
setting.  Let $a_t$ be a piecewise differentiable path in $\Matx$, and
write $a_t = \lambda_t a_0 \exp(a_0^{-1} b_t)$, $t \in [0,1]$, with
$\lambda_t \in \R_+$ and $\tr_{a_0} b_t = 0$ for all $t$.  (Note that
by the discussion following Definition \ref{dfn:5}, this is always
possible.)  First, we claim that the path $\bar{a}_t := a_0
\exp(a_0^{-1} b_t)$ is the projection of $a_t$ onto
$\M_{x,\sqrt{A_0}}$, since
\begin{equation*}
  \sqrt{\bar{A}_t} = \sqrt{A_0} \sqrt{\det \exp(a_0^{-1} b_t)} =
  \sqrt{A_0} \sqrt{\exp(\tr_{a_0} b_t)} = \sqrt{A_0}.
\end{equation*}

Write $b_t$ in polar coordinates with respect to $\langle \cdot, \cdot
\rangle_{\bar{a}_t}$; that is,
\begin{equation*}
  b_t = \beta_t c_t, \qquad c_t := \frac{b_t}{\abs{b_t}_{\bar{a}_t}}.
\end{equation*}
As in Theorem \ref{thm:6}, define
\begin{equation*}
  \hat{\beta}_t := \min \left( \max_{s \in [0,t]} \beta_t, 1 \right)
\end{equation*}
and $\hat{a}_t := \exp^0_{a_0}(\hat{\beta}_t b_1)$, where, as in
Theorem \ref{thm:8}, $\exp^0$ denotes the exponential mapping of
$\M_{x,\sqrt{A_0}}$.  With this notation, we have the following.

\begin{lemma}\label{lem:12}
                  Define
    $\tilde{a}_t := \lambda_t \hat{a}_t$.  Then $\tilde{a}_t$ is a
  rectifiable path from $a_0$ to $a_1$, and $L(\tilde{a}_t) \leq
  L(a_t)$, with equality if and only if $\bar{a}_t$ is a
  reparametrization of the radial geodesic (of $\M_{x,\sqrt{A_0}}$),
  $a^0_t := \exp^0_{a_0} (t b_1)$.
\end{lemma}
\begin{proof}
                  Since $\bar{a}_t$ and $\hat{a}_t$ are the projections onto
  $\M_{x,\sqrt{A_0}}$ of $a_t$ and $\tilde{a}_t$, respectively, we have
  \begin{equation}\label{eq:30}
    \tr_{a_t}(\bar{a}'_t) = \lambda_t^{-2} \tr_{\bar{a}_t}(\bar{a}'_t)
    = 0,
  \end{equation}
  and similarly for $\hat{a_t}$.

  Now, note that $\bar{a}_t = \exp^0_{a_0}(b_t)$ and $\hat{a}_t =
  \exp^0_{a_0}(\hat{\beta}_t b_1)$.
              Thus, the hypotheses of Theorem \ref{thm:6} are satisfied with $N =
  \M_{x,\sqrt{A_0}}$, $v = b_1$, $\rho(t) = \bar{a}_t$, and $\sigma(t)
  = \hat{a}_t$.  In particular, we have $\abs{\hat{a}'_t}_{\hat{a}_t}
  \leq \abs{\bar{a}'_t}_{\bar{a}_t}$ for all $t$, with equality if and
  only if $\bar{a}_t$ is a reparametrization of $a^0_t$.

  Let us again consider the paths $a_t = \lambda_t \bar{a}_t$ and
  $\tilde{a}_t = \lambda_t \hat{a}_t$.  We have
  \begin{equation*}
    a'_t = \lambda'_t \bar{a}_t + \lambda_t \bar{a}'_t = \frac{\lambda'_t}{\lambda_t}
    a_t + \lambda_t \bar{a}'_t,
  \end{equation*}
  and similarly $\tilde{a}'_t = (\lambda'_t / \lambda_t) \tilde{a}_t +
  \lambda_t \hat{a}'_t$.  By \eqref{eq:30} and the discussion
  following Definition \ref{dfn:5}, these decompositions are
  orthogonal, since the first term in each is pure-trace and the
  second is traceless.  Thus we have
  \begin{equation*}
    \abs{a'_t}_{a_t}^2 =
    \left[
      \tr_{a_t} \left(
        \left(
          \frac{\lambda'_t}{\lambda_t} a_t
        \right)^2
      \right)
      + \tr_{\lambda_t \bar{a}_t}((\lambda_t \bar{a}'_t)^2)
    \right]
    \sqrt{\lambda_t \bar{A}_t} = \lambda_t^{n/2}
    \left[
      n
      \left(
        \frac{\lambda'_t}{\lambda_t}
      \right)^2
      + \abs{\bar{a}'_t}_{\bar{a}_t}^2
    \right]
    \sqrt{A_0},
  \end{equation*}
  and since $\abs{\hat{a}'_t}_{\hat{a}_t} \leq
  \abs{\bar{a}'_t}_{\bar{a}_t}$, with equality if and only if
  $\bar{a}_t$ is a reparametrization of $a^0_t$,
  \begin{equation*}
    \abs{\tilde{a}'_t}_{\tilde{a}_t}^2 =
    \lambda_t^{n/2}
    \left[
      n
      \left(
        \frac{\lambda'_t}{\lambda_t}
      \right)^2
      + \abs{\hat{a}'_t}_{\hat{a}_t}^2
    \right] \sqrt{A_0}
    \leq \abs{a'_t}_{a_t}^2.
  \end{equation*}
  The result of the lemma now follows.
\end{proof}

\subsubsection{Minimality of Riemannian geodesics}
\label{sec:minim-riem-geod}

The results of the last subsection allow us to prove, with one
additional estimate, the minimality of Riemannian geodesics in
$\Matx$.

\begin{theorem}\label{thm:9}
  Let $a_0, a_1 \in \Matx$.  Suppose that $a_1 = a_0 \exp(a_0^{-1}
  b)$, where $\tr_{a_0}(b_T^2) < (4 \pi)^2 / n$.  Let $a_t$ be the
  Riemannian geodesic between $a_0$ and $a_1$ as given in Theorems
  \ref{thm:4} and \ref{thm:7}.  Then $a_t$ is the unique minimal path
  (up to reparametrization) in $\overline{\Matx}$ between $a_0$ and
  $a_1$.
\end{theorem}
\begin{proof}
  Since, by Theorem \ref{thm:7}, $\exp_{a_0}$ is a diffeomorphism onto
  its image, Theorem \ref{thm:6} implies that $a_t$ is the unique
  minimal path among the class of paths that lie completely within the
  image of $\exp_{a_0}$.  Therefore, we must show that there are no
  shorter paths that exit $V := \im(\exp_{a_0})$.

  By Proposition \ref{prop:1} and Theorem \ref{thm:7}, the boundary of
  $V$ in $\overline{\Matx}$ is
  \begin{equation}\label{eq:31}
    \partial V = \{ [0] \} \cup
    \left\{
      c \in \Matx \midmid c = a_0
      \exp(a_0^{-1} \gamma),\ \tr_{a_0}(\gamma_T^2) = \frac{(4 \pi)^2}{n}
    \right\}.
  \end{equation}

  So first, suppose that there exists a path $\alpha_t = \lambda_t a_0
  \exp(a_0^{-1} c_t)$ in $\Matx$ between $a_0$ and $a_1$ with
  $L(\alpha_t) \leq L(a_t)$, $\tr_{a_0}(c_t) \equiv 0$, $\lambda_t \in
  \R$, and $\tr_{a_0}(c_{t_0}^2) \geq (4 \pi)^2 / n$ for some $t_0 \in
  (0, 1)$.  Since $\tr_{a_0}(b_T^2) < (4 \pi)^2 / n$, we may deduce
  from Lemma \ref{lem:12} that there exists a path $\tilde{a}_t$ lying
  completely in $V$ with $L(\tilde{a}_t) < L(\alpha_t) \leq L(a_t)$, a
  contradiction to the minimality of $a_t$ among paths lying in $V$.

  Now, consider any path $\alpha_t$ from $a_0$ to $a_1$ that passes
  through $[0]$.  By Proposition \ref{prop:1}, we have that $d_x(a_0,
  [0]) = \frac{4}{\sqrt{n}} \sqrt[4]{A_0}$, and similarly for $a_1$.
  Therefore $L(\alpha_t) \geq \frac{4}{\sqrt{n}} (\sqrt[4]{A_0} +
  \sqrt[4]{A_1})$.  Thus, if we can show that $L(a_t) <
  \frac{4}{\sqrt{n}} (\sqrt[4]{A_0} + \sqrt[4]{A_1})$, the theorem
  will be proved.

  To show this, let $\psi$ be as in Theorem \ref{thm:7}, so that $a_1
  = \exp_{a_0}(\psi(b))$.  Since the case $b_T = 0$ is trivial, we
  simply estimate $L(a_t) = \abs{\psi(b)}_{a_0}$ for $b_T \neq 0$.  We
  first have
  {\allowdisplaybreaks
    \begin{equation}\label{eq:33}
      \begin{aligned}
        \tr_{a_0}(\psi(b)^2) &= \frac{16}{n^2} \left[ \exp \left(
            \frac{\tr_{a_0} b}{2} \right) \cos^2 \left( \frac{\sqrt{n
                \tr_{a_0}(b_T^2)}}{4} \right)
        \right. \\
        &\left. \qquad \qquad {} - 2 \exp \left( \frac{\tr_{a_0} b}{4}
          \right) \cos \left( \frac{\sqrt{n \tr_{a_0}(b_T^2)}}{4}
          \right) + 1
        \right] \tr_{a_0}(a_0^2) \\
        &\qquad {} + \frac{16}{n \tr_{a_0}(b_T^2)} \exp \left(
          \frac{\tr_{a_0} b}{2} \right) \sin^2 \left( \frac{\sqrt{n
              \tr_{a_0}(b_T^2)}}{4} \right)
        \tr_{a_0}(b_T^2) \\
        &= \frac{16}{n} \left( \exp \left( \frac{\tr_{a_0} b}{2}
          \right) - 2 \exp \left( \frac{\tr_{a_0} b}{4} \right) \cos
          \left( \frac{\sqrt{n \tr_{a_0}(b_T^2)}}{4} \right) + 1
        \right).
      \end{aligned}
    \end{equation}
  }On the other hand, using the formula $\exp(\tr_{a_0}(b)) =
  \det(\exp(a_0^{-1} b))$, we have $\exp(\frac{\tr_{a_0} b}{2}) =
  \sqrt{\det \exp(a_0^{-1} b)} = \frac{\sqrt{A_1}}{\sqrt{A_0}}$.  Therefore,
  \begin{equation}\label{eq:39}
    \begin{aligned}
      \abs{\psi(b)}_{a_0}^2 &= \tr_{a_0}(\psi(b)^2) \sqrt{A_0} \\
      &= \frac{16}{n} \left( \frac{\sqrt{A_1}}{\sqrt{A_0}} - 2 \frac{\sqrt[4]{A_1}}{\sqrt[4]{A_0}}
        \cos \left( \frac{\sqrt{n \tr_{a_0}(b_T^2)}}{4} \right) + 1
      \right) \sqrt{A_0} \\
      &= \frac{16}{n}
      \left(
        \sqrt{A_0} - 2 \sqrt[4]{A_0} \sqrt[4]{A_1} \cos \left(
          \frac{\sqrt{n \tr_{a_0}(b_T^2)}}{4} \right) + \sqrt{A_1}
      \right).
    \end{aligned}
  \end{equation}
  Since $\tr_{a_0}(b_T^2) < (4\pi)^2 / n$, the argument of cosine in
  the above equation lies strictly between $0$ and $\pi$, and
  therefore we can estimate
                                                          \begin{equation}\label{eq:34}
    \abs{\psi(b)}_{a_0}^2 <
                                            \frac{16}{n}
    \left(
      \sqrt[4]{A_0} + \sqrt[4]{A_1}
    \right)^2,
  \end{equation}
  as was to be shown.
\end{proof}

\subsubsection{Geodesics in $\overline{\Matx}$}
\label{sec:geod-overl}

Theorem \ref{thm:9} will now allow us to determine that, for any
element $a \in \Matx$, the singular point $[0]$ is the unique closest
point to $a$ on the boundary of the image of $\exp_a$.  This will help
to find geodesics between points that do not have a Riemannian
geodesic connecting them.

\begin{lemma}\label{lem:16}
  Let $a_0 \in \Matx$, and let $a_1 \in \partial V \subset
  \overline{\Matx}$, where $V$ denotes the image of $\exp_{a_0}$
  (cf.~Theorems \ref{thm:7} and \eqref{eq:31}).  Then
  \begin{equation}\label{eq:32}
    d_x(a_0, a_1) = \frac{4}{\sqrt{n}}
    \left(
      \sqrt[4]{A_1} + \sqrt[4]{A_0}
    \right).
  \end{equation}
  In particular, $d_x(a_0, \partial V) = \frac{4}{\sqrt{n}}
  \sqrt[4]{A_0}$, and this distance is realized uniquely for $[0]
  \in \partial V$.
\end{lemma}
\begin{proof}
  If $a_1 = [0]$, then by Proposition \ref{prop:1}, $d_x(a_0, a_1)
  = \frac{4}{\sqrt{n}} \sqrt[4]{A_0}$. Thus \eqref{eq:32} is proved
  in this case.

  If $a_1 \neq [0]$, then we write $a_1 = a_0 \exp(a_0^{-1} b)$ with
  $\tr_{a_0} (b_T^2) = (4 \pi)^2 / n$.  Let $c_k = a_0 \exp(a_0^{-1}
  b_k)$ be a sequence of elements of $V$ with $d_x (c_k, a_1)
  \rightarrow 0$; in particular, $d_x(a_0, c_k) \rightarrow d_x(a_0,
  a_1)$.  By Proposition \ref{prop:1}, $\{ c_k \}$ (resp.~$\{ b_k \}$)
  converges in the standard topology on $\Matx$ to $a_1$ (resp.~$b$).
  In particular, $\tr_{a_0} ((b_k)_T^2) \rightarrow (4 \pi)^2 / n$.
  Since $c_k \in V$, we may conclude
                  \begin{equation*}
    d_x(a_0, a_1)^2 = \lim_{k \rightarrow \infty} d_x(a_0, c_k)^2 =
    \lim_{k \rightarrow \infty} \abs{\psi(b_k)}_{a_0}^2 = 
    \frac{16}{n}
    \left(
      \sqrt[4]{A_1} + \sqrt[4]{A_0}
    \right)^2,
  \end{equation*}
  where we have used \eqref{eq:39} in the last equality.
\end{proof}

We are now in a position to prove the main result of this section.

\begin{theorem}\label{thm:10}
  There exists a unique (up to reparametrization) minimal path between
  any two points $a_0, a_1 \in \overline{\Matx}$.  If there exists $b
  \in \satx$ such that $a_1 = a_0 \exp(a_0^{-1} b)$, then this minimal
  path is given by
  \begin{enumerate}
  \item \label{item:14} the Riemannian geodesic connecting $a_0$ and
    $a_1$ (cf.~Theorems \ref{thm:4} and \ref{thm:7}), if we have
    $\tr_{a_0}(b_T^2) < (4 \pi)^2 / n$;
  \item \label{item:15} the concatenation of the straight segments
    from $a_0$ to $[0]$ and from $[0]$ to $a_1$, if we have $\tr_{a_0}(b_T^2)
    \geq (4 \pi)^2 / n$.
  \end{enumerate}
  Otherwise either $a_0$ or $a_1$ is $[0]$, and the minimal path is
  the straight segment between the two.
\end{theorem}
\begin{proof}
  We have already shown statement \eqref{item:14} in Theorem
  \ref{thm:9}.

  Let us now show that the unique minimal path between $a_0 \in \Matx$
  and $[0]$ is the straight segment $a_t = t a_0$.  In the discussion
  preceding Lemma \ref{lem:9}, we remarked that this segment is
  minimal, so it remains to show uniqueness.  Let $\hat{a}_t = a_0
  \exp(a_0^{-1} b_t)$, $t \in [0,1]$, be a differentiable path with
  $\tr_{a_0} b_t = 0$ for all $t$.  Furthermore, let $\lambda_t$ be a
  differentiable family of nonnegative real numbers with $\lambda_1 =
  0$, and define a path $\tilde{a}_t := \lambda_t \hat{a}_t$ from
  $a_0$ to $[0]$.  (Note that any path from $a_0$ to $[0]$ can be
  written in this way.)  As in the proof of Lemma \ref{lem:12}, we
  have
  \begin{equation*}
    \abs{\tilde{a}'_t}_{a_t}^2 =
    \left[
      n
      \left(
        \frac{\lambda'_t}{\lambda_t}
      \right)^2
      + \abs{\hat{a}'_t}_{\hat{a}_t}^2
    \right].
  \end{equation*}
  This quantity is minimized when (i) $\abs{\hat{a}'_t}_{\hat{a}_t} = 0$,
  i.e., when $b_t = 0$, for all $t$, and (ii) $\lambda'_t \leq 0$ for
  all $t$.  But then $\tilde{a}_t$ is simply a reparametrization of
  $a_t = t a_0$, as desired.

  Finally, to show \eqref{item:15}, we note that any path from $a_0$
  to $a_1$ must necessarily pass through $\partial V$
  (cf.~\eqref{eq:31}).  By the above argument and Lemma \ref{lem:16},
  the unique minimal path between $a_0$ and $\partial V$ is the
  straight segment between $a_0$ and $[0]$.  If additionally $W :=
  \im(\exp_{a_1})$, then we also see that any path from $a_1$ to $a_0$
  must pass through $\partial W \subset \overline{\Matx}$.  Since the
  unique minimal path from $a_1$ to $\partial W$ is the straight
  segment from $a_1$ to $[0]$, and the concatenation of these straight
  segments is continuous and rectifiable, this concatenation is the
  unique minimal path connecting $a_0$ and $a_1$.
\end{proof}

In the two following subsections, we work out the consequences of this
theorem: the continuous dependence of geodesics in $\overline{\Matx}$
on their endpoints, and formulas for geodesics and distance between
elements of the completion $\overline{\M}$ of the manifold of
Riemannian metrics.

\subsection{Continuous dependence of geodesics}
\label{sec:cont-depend-geod}

The goal of this subsection is to prove the following theorem, which
we will require to show the CAT(0) property for $\overline{\Matx}$.

\begin{theorem}\label{thm:12}
  Geodesics in $\overline{\Matx}$ vary continuously with their
  endpoints.
\end{theorem}

Continuous dependence means the following: Let $\alpha_k, \beta_k,
\alpha, \beta \in \overline{\Matx}$, with $\alpha_k \xrightarrow{d_x}
\alpha$ and $\beta_k \xrightarrow{d_x} \beta$.  Let $\gamma_k, \gamma
: [0,1] \rightarrow \overline{\Matx}$ be the geodesics, parametrized
proportionally to arc length, connecting $\alpha_k$ to $\beta_k$ and
$\alpha$ to $\beta$, respectively.  Then $\gamma_k \rightarrow \gamma$
in $C^0([0,1], \overline{\Matx})$, where the metric on
$\overline{\Matx}$ used to define $C^0([0,1], \overline{\Matx})$ is, of
course, $d_x$.  Furthermore, considering the mapping $\overline{\Matx}
\times \overline{\Matx} \rightarrow C^0([0,1], \overline{\Matx})$
given by mapping two points to the geodesic between them, one sees
that it suffices to prove continuity when $\alpha_k$ is fixed, since
this mapping is continuous if and only if it is continuous on each
factor of the domain.

The remainder of this subsection will constitute the proof of Theorem
\ref{thm:12}.
 
Let $a_0, a_1 \in \overline{\Matx}$ be given arbitrarily.  We must
show that if $\{ a_k \}$ is any sequence with $d_x(a_k, a_1)
\rightarrow 0$, then the geodesics $a_{k,t}$ connecting $a_0$ and
$a_k$ converge in the $C^0$ topology to the geodesic $a_t$ connecting
$a_0$ and $a_1$.  We will thus use the description of geodesics in
Theorem \ref{thm:10}.  Note also that by Proposition \ref{prop:1}, if
$a_1 \neq [0]$, then $\{ a_k \}$ converges in the standard topology on
$\Matx$.

We will adopt the convention that if $a_k \neq [0]$, then we write
$a_k = a_0 \exp(a_0^{-1} b_k)$ for some $b_k \in \satx$.  Furthermore,
let $c_k := b_k - \frac{1}{n} \tr_{a_0}(b_k)$ be the $a_0$-traceless
part of $b_k$.  Similarly, if $a_1 \neq [0]$, we may write $a_1 = a_0
\exp(a_0^{-1} b)$, with $c$ the $a_0$-traceless part of $b$.

We now break the proof into several cases.

\subsubsection{The case $a_1 = [0]$}
\label{sec:case-a_1-=}

In this case, assume that $a_k \neq [0]$ for $k$ large enough, since
otherwise the desired result is trivial.  Then by Proposition
\ref{prop:1}, $\sqrt[4]{A_k} \rightarrow 0$, so $\exp \left(
  \frac{1}{4} \tr_{a_0}(b_k) \right) \rightarrow 0$ as well, since
$\sqrt[4]{A_k} = \sqrt[4]{A_0} \left( \det \exp (a_0^{-1} b_k)
\right)^{1/4} = \sqrt[4]{A_0} \exp \left( \frac{1}{4} \tr_{a_0}(b_k)
\right)$.  Thus, we first see by Lemma \ref{lem:9} that if
$\tr_{a_0}(c_k^2) \geq (4 \pi)^2 / n$ and $a_k$ is $d_x$-close to
$[0]$, then $a_{k,t}$---the concatenation of the straight segments
from $a_0$ to $[0]$ and $[0]$ to $a_k$---is close to $a_t$---the
straight segment from $a_0$ to $[0]$.  Second, if $\tr_{a_0}(c_k^2) <
(4 \pi)^2 / n$ and $a_k$ is close to $[0]$, then by \eqref{eq:29},
$\exp_{a_0}^{-1}(b_k)$ is close to $-\frac{4}{n} a_0$.  Since
Riemannian geodesics vary continuously with their initial tangent
vector, we thus have that $a_{k,t}$ is $C^0$-close to $a_t =
\exp_{a_0}\left( -\frac{4}{n} t a_0 \right)$, which is nothing but a
certain parametrization of the straight segment between $a_0$ and
$a_1$.  The statement thus follows for $a_1 = [0]$.

\subsubsection{The case $a_1 \neq [0]$, $\tr_{a_0}(c^2) < (4 \pi)^2 / n$}
\label{sec:case-a_1-neq}

In this case (and all subsequent cases), for $k$ large enough, we must
have $a_k \neq [0]$.  Thus, for $k$ large enough, we can say
$\tr_{a_0}(c_k^2) < (4\pi)^2 / n$ as well, and so $a_t$
(resp.~$a_{k,t}$) is the Riemannian geodesic connecting $a_0$ and
$a_1$ (resp.~$a_k$).  Since Riemannian geodesics vary continuously
with their endpoints, the statement holds in this case.

\subsubsection{The case $a_1 \neq [0]$, $\tr_{a_0}(c^2) > (4 \pi)^2 / n$}
\label{sec:case-a_1-neq-1}

Here again, for $k$ large enough, $\tr_{a_0}(c_k^2) > (4\pi)^2 / n$ as well.
Since the straight segments between $a_k$ and $[0]$ converge to the
straight segment between $a_1$ and $[0]$, the statement also holds
here.

\subsubsection{The case $a_1 \neq [0]$, $\tr_{a_0}(c^2) = (4 \pi)^2 / n$}
\label{sec:case-a_1-neq-2}

This final case is the most involved, and requires some analysis of
the limiting behavior of Riemannian geodesics.

First, we note that $a_t$ is the concatenation of the straight
segments from $a_0$ to $[0]$ and $[0]$ to $a_1$.  Therefore, if $a_k$
is close to $a_1$ and $\tr_{a_0}(c_k^2) \geq (4 \pi)^2 / n$, then the
geodesic between $a_0$ and $a_k$ is $C^0$-close to the geodesic
between $a_0$ and $a_1$, by the same argument as in the last case.

Therefore, we must only worry about the case that $\tr_{a_0}(c_k^2) <
(4 \pi)^2 / n$.  For simplicity, we assume that this holds for all $k$
large enough, and the general case follows by combining the previous
paragraph with the following argument.

Let $\tilde{a}_{k,t} := a_{k,1-t}$ be the geodesic from $a_k$ to
$a_0$.  It suffices to show that for all $\delta, \epsilon > 0$, there
exist $\rho, \sigma < 0$ such that for $k$ large enough, we have
$\min_{t \in [0,1]} \sqrt[4]{A_{k,t}} \leq \delta$, $\abs{a'_{k,0} -
  \rho a_0}_{a_0} < \epsilon$, and $\abs{\tilde{a}'_{k,0} - \sigma
  a_k}_{a_k} < \epsilon$.  For in this case, by Theorem \ref{thm:4}
and the continuous dependence of Riemannian geodesics on initial data,
up to the first time $t_0$ (resp.~$t_1$) when $\sqrt[4]{A_{k,t}} =
\delta$ (resp.~$\sqrt[4]{\tilde{A}_{k,t}} = \delta$), we will have
that $a_{k,t}$ is arbitrarily $C^0$-close to the straight segment
between $a_0$ (resp.~$a_k$) and $[0]$.  Furthermore, we can make $k$
large enough that $\sqrt[4]{A_{t_0}}, \sqrt[4]{A_{t_1}} \leq 2
\delta$.  Thus, by Lemma \ref{lem:9}, for $t \in (t_0, t_1)$ we have
\begin{align*}
  d_x(a_{k,t}, a_t) &\leq d_x(a_{k,t}, [0]) + d_x([0], a_{k,t_0}) + d_x(a_{k,t_0},
  a_{t_0}) + d_x(a_{t_0}, [0]) + d_x([0], a_t) \\
  &< \frac{4}{\sqrt{n}}
  \delta + \frac{4}{\sqrt{n}} \delta + d_x(a_{k,t_0},
  a_{t_0}) + 2 \cdot \frac{4}{\sqrt{n}} \delta + 2 \cdot \frac{4}{\sqrt{n}} \delta,
\end{align*}
which is arbitrarily small.  (Recall from Theorem \ref{thm:4} that the
change in the volume element is quadratic, so we have
$\sqrt[4]{A_{k,t}} < \delta$ if $t \in (t_0, t_1)$.  Also, it is clear
from the form of $a_t$ that $\sqrt[4]{A_t} < 2 \delta$ for $t \in
(t_0, t_1)$.)

We first show the statement about $a'_{k,0}$.  To start, note that
$b_k \rightarrow b$, and in particular $\tr_{a_0}(b_k) \rightarrow
\tr_{a_0} b$ and $\tr_{a_0}(c_k^2) \rightarrow \tr_{a_0}(c^2) = (4
\pi)^2 / n$.  From \eqref{eq:29}, we see that
\begin{equation}\label{eq:35}
  \begin{aligned}
    a'_{k,0} = \exp_{a_0}^{-1}(a_k) &= \frac{4}{n} \left( \exp \left(
        \frac{\tr_{a_0} b_k}{4} \right) \cos \left( \frac{\sqrt{n
            \tr_{a_0}(c_k^2)}}{4} \right) -1 \right)
    a_0 \\
    &\quad + \frac{4}{\sqrt{n \tr_{a_0}(c_k^2)}} \exp \left(
      \frac{\tr_{a_0} b_k}{4} \right) \sin \left( \frac{\sqrt{n
          \tr_{a_0}(c_k^2)}}{4} \right) c_k.
  \end{aligned}
\end{equation}
By the above arguments, the factors of cosine and sine in
\eqref{eq:35} converge to $-1$ and 0, respectively.  Thus, $a'_{k,0}
\rightarrow \rho a_0$ for some $\rho < 0$, as desired.

Similarly, note that $a_0 = a_k \exp(-a_0^{-1} b_k) = a_k
\exp(a_k^{-1} (-a_k a_0^{-1} b_k))$.  We also have
\begin{gather*}
  \tr_{a_k}(-a_k a_0^{-1} b_k) = -\tr_{a_0}(b_k), \qquad
  \tr_{a_k}(-a_k a_0^{-1} c_k) = \tr_{a_0}(-c_k) = 0, \\
  \tr_{a_k}((-a_k a_0^{-1} c_k)^2) = \tr_{a_0}(c_k^2).
\end{gather*}
  Thus, in this case, \eqref{eq:29} gives
\begin{equation}\label{eq:36}
  \begin{aligned}
    \tilde{a}'_{k,0} = \exp_{a_k}^{-1}(a_0) &= \frac{4}{n} \left( \exp
      \left( -\frac{\tr_{a_0} b_k}{4} \right) \cos \left(
        \frac{\sqrt{n \tr_{a_0}(c_k^2)}}{4} \right) -1 \right)
    a_k \\
    &\quad - \frac{4}{\sqrt{n \tr_{a_0}(c_k^2)}} \exp \left(
      -\frac{\tr_{a_0} b_k}{4} \right) \sin \left( \frac{\sqrt{n
          \tr_{a_0}(c_k^2)}}{4} \right) a_k a_0^{-1} c_k.
  \end{aligned}
\end{equation}
The same argument as for $a'_{k,0}$ shows that there exists $\sigma <
0$ for which $\tilde{a}'_{k,0}$ is arbitrarily close to $\sigma a_k$
for $k$ large enough.

To conclude the proof of this case, we show the statement about
$\sqrt[4]{A_{k,t}}$.  Recall from Theorem \ref{thm:4} that
$\sqrt{A_{k,t}} = (q_k(t)^2 + r_k(t)^2) \sqrt{A_0}$, where $q_k(t) = 1
+ \frac{t}{4} \tr_{a_0}(a'_{k,0})$ and $r_k(t) = \frac{t}{4}\sqrt{n
  \tr_{a_0}((a'_{k,0})_T^2)}$.  Note that $\min_{t \in [0,1]} q_k(t)^2 =
0$ for all $k$, and that by the above arguments,
$\tr_{a_0}((a'_{k,0})_T^2) \rightarrow 0$ for $k \rightarrow \infty$.
These facts combine to show that $\min_{t \in [0,1]} \sqrt[4]{A_{k,t}}
\leq \delta$ for $k$ large enough.

Thus, Theorem \ref{thm:12} is proved.

\subsection{Geodesics in $\overline{\M}$}
\label{sec:geodesics-overlinem}

With minimal paths in $(\Matx, d_x)$ determined, we can explicitly
determine $d_x$ and thus, by Theorem \ref{thm:2}, $d$.  Furthermore,
since there exists a unique minimal path between any two elements in
$\Matx$, one sees that there is a unique minimal path between any two
elements $g_0, g_1 \in \overline{\M}$: the path $g_t$ that gives the
minimal path $g_t(x)$ between $g_0(x)$ and $g_1(x)$ for each $x \in
M$.  We thus are now able to combine these results into a theorem for
$\overline{\M}$, as well as summarize the explicit realizations of
geodesics and distance that we have determined up to this point,
reformulating them for $\overline{\M}$.  (Note that to get the
statement of the theorem we use, in addition to the theorems of the
preceding subsections, the formula \eqref{eq:39}.)

\begin{theorem}\label{thm:3}
  There exists a unique minimal path $g_t$, $t \in [0,1]$, between any
  two points $g_0, g_1 \in \overline{\M}$, given by the following.
  Let $k$ be a measurable, symmetric $(0,2)$-tensor field on $M$,
  defined on the subset $N$ where neither $g_0$ nor $g_1$ is zero,
  such that $g_1 = g_0 \exp(g_0^{-1} k)$ on this subset.  Denote by $P$
  the subset of $M$ where $\tr_{g_0}(k_T^2) < (4\pi)^2 / n$ (here
  $k_T$ denotes the traceless part of $k$).  Write $h := \psi(k)$,
  where $\psi(k)(x)$ is given as in Theorem \ref{thm:7}.  Finally, let
  $q_t$ and $r_t$ be one-parameter families of functions on $N$ given
  by
  \begin{equation*}
    q_t := 1 + \frac{t}{4} \tr_{g_0}(h), \quad r_t :=
    \frac{t}{4} \sqrt{n \tr_{g_0}(h_T^2)}.
  \end{equation*}

  Then at points $x \in N \cap P$ where $h_T(x) \neq 0$, we have
  \begin{equation*}
    g_t(x) = 
    \left( q_t^2(x) + r_t^2(x) \right)^{\frac{2}{n}} g_0(x) \exp
    \left( \frac{4}{\sqrt{n \tr_{g_0(x)}(h_T^2(x))}} \arctan \left(
        \frac{r_t(x)}{q_t(x)} \right) g_0^{-1}(x) h_T(x) \right).
  \end{equation*}
  At points $x \in N \cap P$ where $h_T(x) = 0$, we have
  \begin{equation*}
    g_t(x) = q_t^{4/n} g_0(x) =
    \left(
      1 + \frac{\sqrt[4]{G_1} - \sqrt[4]{G_0}}{\sqrt[4]{G_0}} t
    \right)^{\frac{4}{n}} g_0(x)
  \end{equation*}
  In both of the above cases $g_t(x)$ does not intersect $[0]$.
  Additionally, the range of $\arctan$ is given by the following.  At
  a point where $\tr_{g_0} h \geq 0$, it assumes values in
  $(-\frac{\pi}{2}, \frac{\pi}{2})$.  At a point where $\tr_{g_0} h <
  0$, $\arctan(r(t) / q(t))$ assumes values as follows, with $t_0 := -
  \frac{4}{\tr_{g_0} h}$:
  \begin{enumerate}
  \item in $[0, \frac{\pi}{2})$ if $0 \leq t < t_0$,
  \item in $(\frac{\pi}{2}, \pi)$ if $t_0 < t < \infty$,
  \end{enumerate}
  and we set $\arctan(r(t) / q(t)) = \frac{\pi}{2}$ if $t = t_0$.

  At all other points of $M$, $g_t(x)$ passes through $[0]$, and we
  have
  \begin{equation*}
    g_t(x) =
    \begin{cases}
      \left( 1 - \frac{\sqrt[4]{G_0(x)} + \sqrt[4]{G_1(x)}}{\sqrt[4]{G_0(x)}} t
        \right)^{\frac{4}{n}} g_0(x), & t \in \left[ 0,
          \frac{\sqrt[4]{G_0(x)}}{\sqrt[4]{G_0(x)} + \sqrt[4]{G_1(x)}} \right],
      \\[1em]
      \left(
        \frac{\sqrt[4]{G_0(x)} + \sqrt[4]{G_1(x)}}{\sqrt[4]{G_1(x)}}
        t - \frac{\sqrt[4]{G_0(x)}}{\sqrt[4]{G_1(x)}}
      \right)^{\frac{4}{n}}
      g_1(x),
      & t \in \left[ \frac{\sqrt[4]{G_0(x)}}{\sqrt[4]{G_0(x)} +
           \sqrt[4]{G_1(x)}}, 1 \right].
    \end{cases}
  \end{equation*}

  The distance induced by the $L^2$ Riemannian metric between $g_0$
  and $g_1$ is given by
  \begin{equation*}
    d(g_0, g_1) = \Omega_2(g_0, g_1) = \left( \integral{M}{}{d_x(g_0(x),
      g_1(x))^2}{d \mu} \right)^{1/2},
  \end{equation*}
  with
  \begin{equation*}
    d_x(g_0(x), g_1(x)) =
    \begin{cases}
      \abs{\psi(k)(x)}_{g_0(x)}, & x \in N \cap P,\ k_T \neq 0, \\
      \frac{4}{\sqrt{n}} \abs{\sqrt[4]{G_1(x)} - \sqrt[4]{G_0(x)}}, &
      x \in N \cap P,\ k_T = 0, \\
      \frac{4}{\sqrt{n}}
      \left(
        \sqrt[4]{G_0(x)} + \sqrt[4]{G_1(x)}
      \right), & x \not\in N \cap P.
    \end{cases}
  \end{equation*}
  Here, $\abs{\psi(k)(x)}_{g_0(x)}$, for $k_T \neq 0$, is given
  explicitly by
  \begin{equation*}
    \abs{\psi(k)(x)}_{g_0(x)} =
    \frac{4}{\sqrt{n}}
    \left(
      \sqrt{G_0(x)} - 2 \sqrt[4]{G_0(x)} \sqrt[4]{G_1(x)} \cos \left(
        \frac{\sqrt{n \tr_{g_0}(k_T(x)^2)}}{4} \right) + \sqrt{G_1(x)}
    \right)^{1/2}
                                                              \end{equation*}
\end{theorem}

As noted in the introduction, the existence of geodesics in
$\overline{\M}$ is in stark contrast to the situation for the
incomplete space $\M$, where the image of the geodesic mapping at a
point contains no open $d$-ball.

\section{The CAT(0) property}
\label{sec:cat0-property}

In this section, we show the CAT(0) property for $\overline{(\M, d)}$.
We will require all of the main results that we have shown so far.

The strategy is to first show that $\overline{(\Matx, d_x)}$ is
CAT(0), which will follow from several general theorems on CAT(0)
spaces, as well as a further technical lemma using the results of the
last section to show that the interiors of certain geodesic triangles
in $\overline{\Matx}$ are totally geodesic.

Following that, we will apply a concise argument to show that the
formula $d = \Omega_2$ implies that $\overline{\M}$ is a CAT(0) space
if $\overline{\Matx}$ is for each $x \in M$.  This will give the main
result.

\subsection{Background and notation}
\label{sec:background-notation}

We introduce the following conventions and notation in this section.
For any $a, b \in \Matx$, denote by $v_{a,b}$ the unique element of
$T_a \Matx \cong \satx$ such that $b = a \exp(a^{-1} v_{a,b})$.  Let
$u_{a,b} := v_{a,b} - \frac{1}{n} (\tr_a v_{a,b}) a$ denote the traceless
part of $v_{a,b}$.

If $a, b \in \overline{\Matx}$, then we denote by $[a,b]$ the unique
(unparametrized) geodesic connecting $a$ and $b$.

Let us now discuss some facts from metric geometry, in particular the
definition of a CAT(0) space.  Our main reference here is
\cite{Bridson:1999tu}.

Let $(X, \delta)$ be a path metric space.  Let $x, y, z \in X$, and
let $\triangle xyz$ be a geodesic triangle, that is, a triangle
composed of geodesics (globally distance-minimizing paths) $[x,y]$,
$[y, z]$, and $[z, x]$.  A \emph{comparison triangle} $\triangle
\bar{x} \bar{y} \bar{z}$ for $\triangle xyz$ is a triangle in
Euclidean space, $\Etwo$, with side lengths equal to the side lengths
of $\triangle xyz$.  The points $\bar{x}$, $\bar{y}$, and $\bar{z}$
are called \emph{comparison points}.  More generally, for any $w \in
\triangle xyz$---say $w \in [x,y]$---a comparison point $\bar{w} \in
\triangle \bar{x} \bar{y} \bar{z}$ for $w$ is the unique point on
$[\bar{x}, \bar{y}]$ with $\abs{\bar{w} - \bar{x}} = \delta(x, w)$ (and
consequently, $\abs{\bar{w} - \bar{y}} = \delta(y, w)$).

We say that $X$ is \emph{CAT(0)} if two conditions hold.  First, we
require that there exists a geodesic (minimal path) between any two
points in $X$.  Second, if $x, y, z \in X$ and $\triangle xyz$ is a
geodesic triangle, then for any $w_1, w_2 \in \triangle xyz$ with
comparison points $\bar{w}_1, \bar{w}_2 \in \triangle \bar{x} \bar{y}
\bar{z}$, we have
\begin{equation*}
  d(w_1, w_2) \leq \delta(\bar{w}_1, \bar{w}_2).
\end{equation*}
That is, triangles in $X$ are ``no thicker'' than in Euclidean space.

The space $X$ is said to have \emph{nonpositive curvature} if it is locally a
CAT(0) space, that is, if for each $x \in X$ there exists an open metric
ball $B(x, r_x)$, $r_x > 0$, such that $B(x, r_x)$ with the induced
metric is a CAT(0) space.

There is also a characterization, due to Alexandrov, of CAT(0) spaces
in terms of angles.  Let $x \in X$, and let $\gamma_1, \gamma_2 :
[0,1] \rightarrow X$ be geodesics, parametrized proportional to arc
length, with $\gamma_1(0) = x = \gamma_2(0)$.  Define the
\emph{Alexandrov} (or \emph{upper}) \emph{angle} between $\gamma_1$
and $\gamma_2$ to be
\begin{equation*}
  \angle (\gamma_1, \gamma_2) := \limsup_{s, t \to 0}
  \overline{\angle}_{x} (\gamma_1(s), \gamma_2(t)),
\end{equation*}
where $\overline{\angle}_{x} (\gamma_1(s), \gamma_2(t))$ denotes the
angle in a comparison triangle for $\triangle x \gamma_1(s)
\gamma_2(t)$ at the vertex corresponding to $x$.  Similarly, if
$\triangle xyz \subseteq X$ is a geodesic triangle, then the
Alexandrov angle at the vertex $x$ is the Alexandrov angle, as defined
above, between the geodesics $[x, y]$ and $[x, z]$.

We then have the following alternative characterization of CAT(0)
spaces.

\begin{theorem}[{\cite[Thm.~II.1.7]{Bridson:1999tu}}]\label{thm:15}
  Let $(X, \delta)$ be a metric space such that geodesics exist
  between any two points.  Then the following are equivalent:
  \begin{enumerate}
  \item $X$ is a CAT(0) space.
  \item The Alexandrov angle at any vertex of any geodesic triangle in
    $X$ with distinct vertices is no greater than the angle at the
    corresponding vertex of a comparison triangle in $\Etwo$.
  \end{enumerate}
\end{theorem}

It is this second characterization of CAT(0) spaces that we will use
to show that $\overline{\Matx}$ is CAT(0).  Roughly, it says that a
space is CAT(0) if geodesics diverge at least as quickly as lines in
$\Etwo$.

We will also need a technical lemma that says that if two triangles
have Alexandrov angles no greater than a comparison triangle, and they
can be glued along one side to form a larger triangle, then this
larger triangle also has Alexandrov angles no greater than a
comparison triangle.

\begin{lemma}[{Gluing Lemma for Triangles
    \cite[Lem.~II.4.10]{Bridson:1999tu}}]\label{lem:13}
  Let $X$ be a metric space in which every pair of points can be
  connected by a geodesic.  Let $\triangle p q_1 q_2$ be a triangle in
  $X$ with distinct vertices, and let $r \in [q_1, q_2]$ be distinct
  from $q_1$ and $q_2$.

  If, for both $i = 1, 2$, the angles of $\triangle p q_i r$ are no
  greater than the corresponding angles of a comparison triangle in
  $\Etwo$, then the angles of $\triangle p q_1 q_2$ are also no
  greater than those of a comparison triangle in $\Etwo$.
\end{lemma}

To conclude the background we need, we note that if $a \in \Matx$,
then the Alexandrov and Riemannian angles between geodesic rays based
at $a$ coincide \cite[Cor.~II.1A.7]{Bridson:1999tu}.  Therefore, in
this case, we shall simply refer to the angle between two geodesics.

\subsection{The CAT(0) property for $\overline{(\Matx, d_x)}$}
\label{sec:cat0-prop-overl}

We first wish to show that $\overline{(\Matx, d_x)}$ is a space of
nonpositive curvature in the sense of Alexandrov.  Since the sectional
curvature of $(\Matx, \langle \cdot, \cdot \rangle)$ is nonpositive,
the Riemannian space $(\Matx, d_x)$ has nonpositive curvature in the
sense of Alexandrov \cite[Thm.~II.1A.6]{Bridson:1999tu}.  Therefore,
it only remains to be shown that there is a neighborhood of $[0]$ in
$\overline{(\Matx, d_x)}$ that is a CAT(0) space.  For this, it
suffices to show that any triangle intersecting $[0]$ has (Alexandrov)
angles no greater than the angles in a comparison triangle.  We also
need only consider nontrivial triangles, where none of the vertices
lie on any of the edges, since such a triangle has Alexandrov angles
zero at each vertex.

Consider a triangle $\triangle abc \subset \overline{\Matx}$, with
(Alexandrov) angles $\alpha$, $\beta$, and $\gamma$ at the (distinct)
vertices $a$, $b$, and $c$, respectively.  Assume that $[0] \in
\triangle abc$.  By renaming the vertices if necessary, $\triangle
abc$ then falls into one of the following five distinct cases:
\begin{enumerate}
\item \label{item:8} None of $a, b, c$ are equal to $[0]$, and
  $\tr_a(u_{a,b}^2), \tr_b(u_{b,c}^2), \tr_c(u_{c,a}^2) \geq
  (4\pi)^2/n$.  (See the beginning of \S \ref{sec:background-notation}
  for the definitions of $u_{a,b}$, etc.)
\item \label{item:9} None of $a, b, c$ are equal to $[0]$, and
  $\tr_a(u_{a,b}^2) < (4\pi)^2 / n$, but $\tr_b(u_{b,c}^2),
  \tr_c(u_{c,a}^2) \geq (4\pi)^2/n$.
\item \label{item:10} None of $a, b, c$ are equal to $[0]$, and
  $\tr_a(u_{a,b}^2), \tr_b(u_{b,c}^2) < (4\pi)^2 / n$, but
  $\tr_c(u_{c,a}^2) \geq (4\pi)^2/n$.
\item \label{item:12} We have $c = [0]$ and $\tr(u_{a,b}^2) \geq (4\pi)^2/n$.
\item \label{item:13} We have $c = [0]$ and $\tr(u_{a,b}^2) < (4\pi)^2/n$.
\end{enumerate}
We wish to show that all of these cases are either trivial or reduce
to case \eqref{item:13}.  The cases are depicted in Figure
\ref{fig:all-cases}, where curves with one dash represent the geodesic
$[a,b]$, with two dashes $[b,c]$, and with three $[c,a]$.

\begin{figure}[ht]
  \centering
  \def\svgwidth{0.9\textwidth}
  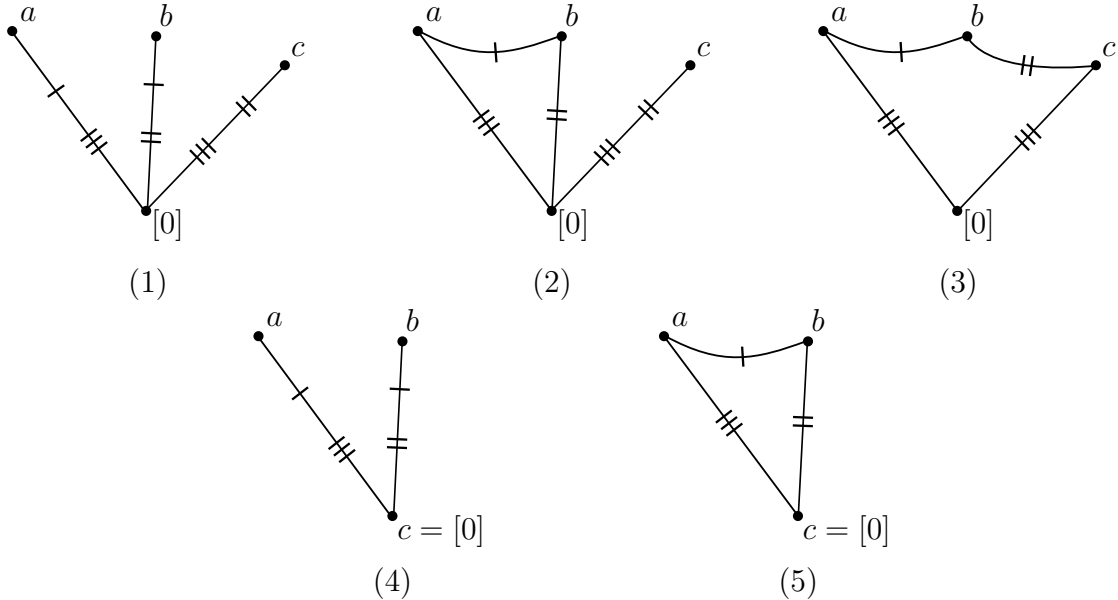
  \caption{Triangles intersecting $[0]$}\label{fig:all-cases}
\end{figure}

Let $\triangle \bar{a} \bar{b} \bar{c}$ be a comparison triangle in
$\mathbb{E}^2$ for $\triangle abc$, with comparison angles
$\bar{\alpha}$, $\bar{\beta}$, and $\bar{\gamma}$.

In case \eqref{item:8}, we immediately see that all of the angles of
$\triangle abc$ are zero.  Therefore, we trivially have $\alpha \leq
\bar{\alpha}$, $\beta \leq \bar{\beta}$, $\gamma \leq \bar{\gamma}$.

In case \eqref{item:9}, we note that $\gamma = 0$ and the angles
$\alpha$ and $\beta$ are the same as those in $\triangle ab[0]$.  On
the other hand, we have that $d_x(a, [0]) < d_x(a, c)$ and $d_x(b,
[0]) < d_x(b, c)$, so the law of cosines (see also Figure
\ref{fig:case-2}) implies that if the Alexandrov angles at the
vertices $a$ and $b$ of $\triangle ab[0]$ are smaller than the
corresponding angles in $\triangle \bar{a} \bar{b} \overline{[0]}$,
then the same holds for $\triangle abc$ and $\triangle
\bar{a}\bar{b}\bar{c}$.  Thus this case reduces to case
\eqref{item:13}.

{\centering
  \begin{figure}[ht]
    \begin{minipage}[b]{0.4\linewidth}
      \centering
      \def\svgwidth{0.75\textwidth}
      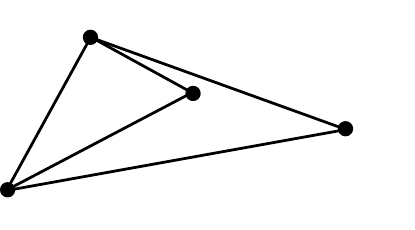
      \caption{Case \eqref{item:9}.}
      \label{fig:case-2}
    \end{minipage}
    \hspace{0.1\linewidth}
    \begin{minipage}[b]{0.4\linewidth}
      \centering
      \def\svgwidth{0.75\textwidth}
      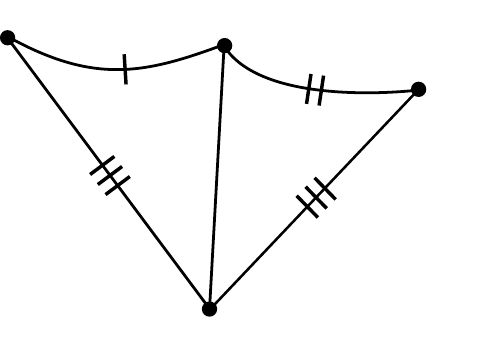
      \caption{Case \eqref{item:10}}
      \label{fig:case-3}
    \end{minipage}
  \end{figure}
}

In case \eqref{item:10}, we may apply the Gluing Lemma for Triangles
\ref{lem:13} with $p = b$, $q_1 = a$, $q_2 = c$, and $r = [0]$ (see
Figure \ref{fig:case-3}) to see that if the angles of any triangle as
in \eqref{item:13} are smaller than those of a comparison triangle,
then so are those of $\triangle abc$.

In case \eqref{item:12}, the (Alexandrov) angles in $\triangle abc$
are zero, as in case \eqref{item:8}, so the angle inequalities are
again trivially satisfied.  (Note that this triangle actually is
trivial, since $[0]$ lies on the edge $[a,b]$, but we include this
case for the sake of clarity.)

Finally, we must deal with case \eqref{item:13}.  For this, we require a
lemma.

\begin{lemma}\label{lem:11}
  Let $a$ and $b$ be as in case \eqref{item:13} above.  Then
  $\triangle ab[0] \subset P$, where $P \subset \satx$ is the plane
  \begin{equation*}
    P := \{y_1 a \exp(y_2 a^{-1} u_{a,b}) \mid (y_1, y_2) \in \R^2 \}.
  \end{equation*}
  Furthermore, the interior $O_{a,b} \subset P$ of $\triangle ab[0]$ is
  a totally geodesic local submanifold of $\Matx$.
\end{lemma}
\begin{proof}
  It is clear that $[a, [0]]$ and $[b, [0]]$ lie in $P$, since these
  are segments of the lines defined by $y_2 = 0$ and $y_2 = 1$,
  respectively.  Furthermore, one sees from \eqref{eq:29} that the
  traceless part of $\exp_a^{-1}(b) = \psi(v_{a,b})$ is proportional
  to $u_{a,b}$.  Therefore, referring to \eqref{eq:28}, it is apparent
  that the geodesic $[a, b]$ is of the following form. We can find a
  pair of functions $(f_1, f_2) : [0,1]^2 \rightarrow \R_+ \times \R$
  such that it is given by $f_1(t) a \exp(f_2(t) u_{a,b})$.  Hence
  $[a,b] \subset P$, so $\triangle ab[0] \subset P$, as desired.

  That $O_{a,b}$ is a local submanifold of $\Matx$ follows immediately
  from the fact that it is an open subset in the submanifold $P \cap
  \Matx$.

  Finally, we show that $O_{a,b}$ is totally geodesic.  Let $\tilde{a}
  = y_1 a \exp(y_2 a^{-1} u_{a,b})$ and $\hat{a} = z_1 a \exp(z_2
  a^{-1} u_{a,b})$ for some $(y_1, y_2), (z_1, z_2) \in \R_+ \times
  \R$.  If we assume $\tilde{a}, \hat{a} \in O_{a,b}$, then we have $0
  < y_2, z_2 < 1$.

  Now we write $a = y_1^{-1} \tilde{a} \exp(-y_2 a^{-1}
  u_{a,b})$ and
  \begin{equation*}
    \hat{a} = \frac{z_1}{y_1} \tilde{a} \exp((z_2 - y_2) a^{-1}
    u_{a,b}) = \frac{z_1}{y_1} \tilde{a} \exp(
    \tilde{a}^{-1} ((z_2 - y_2) \tilde{a} a^{-1} u_{a,b})).
  \end{equation*}
  Note that $\tr_{\tilde{a}}(\tilde{a} a^{-1} u_{a,b}) =
  \tr_a(u_{a,b}) = 0$.  Furthermore,
  \begin{equation*}
    \tr_{\tilde{a}}(((z_2 - y_2)
    \tilde{a} a^{-1} u_{a,b})^2) = (z_2 - y_2)^2 \tr_a(u_{a,b}^2) < \frac{(4\pi)^2}{n},
  \end{equation*}
  since $z_2 < 1$, $y_2 > 0$, and by assumption,
  $\tr_a(u_{a,b}^2) < (4\pi)^2 / n$.  Thus, $\hat{a} \in
  \im(\exp_{\tilde{a}})$.
  
  Let $\psi := \exp_{\tilde{a}}^{-1}$, and let $c := \psi(\hat{a})$.
  From the above discussion, and recalling \eqref{eq:29}, we note that
  $c_T$ is proportional to $\tilde{a} a^{-1} u_{a,b}$, say $c_T =
  \lambda \tilde{a} a^{-1} u_{a,b}$.  Let $\gamma_t$ be the geodesic
  segment between $\tilde{a}$ and $\hat{a}$.  As above, we can find a
  pair of functions $(f_1, f_2) : [0,1]^2 \rightarrow \R_+ \times \R$
  such that
  \begin{align*}
    \gamma_t &= f_1(t) \tilde{a} \exp(f_2(t) \tilde{a}^{-1} c_T) =
    f_1(t) \tilde{a} \exp(\lambda f_2(t) \tilde{a}^{-1} (\tilde{a}
    a^{-1} u_{a,b})) \\
    &= y_1 f_1(t) a \exp((y_2 + \lambda f_2(t)) a^{-1} u_{a,b}).
  \end{align*}
  Thus $\gamma_t$ is contained in $P$.

  To see that $\gamma_t$ is contained in $O_{a,b}$, we first note that
  minimality of $\gamma_t$ and Lemma \ref{lem:12} imply that $f_2$ is
  a monotone function.  Thus, if $\gamma_t$ exits $O_{a,b}$, then
  there must exist $t_0 \in (0,1)$ with $\gamma_{t_0} \in [a,b]$.
  Since distinct geodesics cannot intersect tangentially, there must
  be $t_1 > t_0$ with $\gamma_{t_1} \in [a,b]$, and $\gamma_t \not\in
  [a,b]$ for $t \in (t_0, t_1)$.  But minimality of $\gamma_t$ and
  $[a,b]$ then imply the existence of two distinct minimal paths
  between $\gamma_{t_0}$ and $\gamma_{t_1}$, contradicting Theorem
  \ref{thm:10}.
        This completes the proof.
\end{proof}

We now return to case \eqref{item:13}.  Since $O_{a,b}$ is a totally
geodesic local submanifold of a manifold with nonpositive Riemannian
curvature, it is itself a space of nonpositive curvature (in the sense
of Alexandrov) \cite[Thm.~II.1A.6]{Bridson:1999tu}.  By Theorem
\ref{thm:10} and Lemma \ref{lem:11}, there exists a unique minimal
path between any two points in $O_{a,b}$.  Furthermore,
by Theorem \ref{thm:12}, minimal paths between points in $O_{a,b}$
vary continuously with their endpoints.  Thus, we may apply a theorem
of Alexandrov \cite[Thm.~II.4.9]{Bridson:1999tu} to see that $O_{a,b}$
is a CAT(0) space.

Consider the completion (w.r.t.~$d_x$), $\overline{O_{a,b}} = O_{a,b}
\cup \triangle ab[0]$.  By \cite[Cor.~II.3.11]{Bridson:1999tu}, the
completion of a CAT(0) space is CAT(0), so $\overline{O_{a,b}}$ is
CAT(0).  In particular, $\triangle ab[0] \subset \overline{O_{a,b}}$
satisfies the CAT(0) inequality, and so has Alexandrov angles no
greater than a comparison triangle $\triangle \bar{a} \bar{b}
\overline{[0]}$, as was to be shown.

This completes all possible cases, and so we have the following.

\begin{theorem}\label{thm:11}
  $\overline{(\Matx, d_x)}$ is a space of nonpositive curvature in the
  sense of Alexandrov.
\end{theorem}

Furthermore, by Theorem \ref{thm:10}, there exists a unique minimal
geodesic between any two points in $\overline{\Matx}$, and by Theorem
\ref{thm:12}, these geodesics vary continuously with their endpoints.
Thus, we may again apply the aforementioned theorem of Alexandrov
\cite[Thm.~II.4.9]{Bridson:1999tu} to immediately obtain:

\begin{theorem}\label{thm:13}
  $\overline{(\Matx, d_x)}$ is a CAT(0) space.
\end{theorem}

\subsection{The CAT(0) property for $\overline{\M}$}
\label{sec:cat0-prop-overl-1}

We now claim that Theorems \ref{thm:3} and \ref{thm:13} imply that
$\overline{(\M, d)}$ is a CAT(0) space as well.  Theorem \ref{thm:3}
gives the existence of geodesics, so it remains to show that triangles
in $\overline{\M}$ satisfy the CAT(0) inequality.

To see this, consider the product bundle $M \times \Etwo$, where
$\Etwo = (\R^2, \absdot)$, as above, denotes Euclidean space with its
standard norm.  We denote the fiber over $x$ of this bundle by
$\Etwo_x$.  Our reference volume form $\mu$ on $M$ allows us to define
a Hilbert space $L^2(M \times \Etwo, \mu)$ of square-integrable
sections of this bundle.  As a Hilbert space, it is flat; in
particular, it is CAT(0) \cite[p.~167]{Bridson:1999tu}.  Thus, to show
that triangles in $\overline{\M}$ satisfy the CAT(0) inequality, it
suffices to take a comparison triangle in $L^2(M \times \Etwo, \mu)$
rather than in $\Etwo$.

Let $g, h, k \in \overline{\M}$, and let $\ell_1, \ell_2 \in \triangle
g h k$ be any points.  Assume, without loss of generality, that
$\ell_1 \in [g, h]$, the geodesic between $g$ and $h$, and $\ell_2 \in
[g, k]$.  By Theorem \ref{thm:3}, when evaluated at the point $x \in
M$, the geodesic $[g, h]$ in $\overline{\M}$ agrees with the geodesic
$[g(x), h(x)]$ in $\overline{\Matx}$.  Thus $\ell_1(x) \in [g(x),
h(x)]$ for all $x$, and similarly for $\ell_2(x)$.  Define $\bar{g},
\bar{h}, \bar{k}, \bar{\ell}_1, \bar{\ell}_2 \in L^2(M \times \Etwo,
\mu)$ by the following.  For all $x$, let $\bar{g}(x) := (0, 0)$ and
$\bar{h}(x) := (d_x(g(x), h(x)), 0)$.  Let $\bar{k}(x)$ be the unique
point in the closed upper half-plane with $\abs{\bar{k}(x) -
  \bar{g}(x)} = d_x(g(x), k(x))$ and $\abs{\bar{k}(x) - \bar{h}(x)} =
d_x(h(x), k(x))$.  From this description and Theorem \ref{thm:3}, one
then sees that $\triangle \bar{g}(x) \bar{h}(x) \bar{k}(x)$
(resp.~$\triangle \bar{g} \bar{h} \bar{k}$) is a comparison triangle
for $\triangle g(x) h(x) k(x)$ (resp.~$\triangle g h k$).  Finally,
define $\bar{\ell}_1(x)$ (resp.~$\bar{\ell}_2(x)$) to be the unique
point on the segment between $\bar{g}(x)$ and $\bar{h}(x)$
(resp.~$\bar{k}(x)$) with $\abs{\bar{\ell}_1(x) - \bar{g}(x)} =
d_x(g(x), \ell_1(x))$ (resp.~$\abs{\bar{\ell}_2(x) - \bar{g}(x)} =
d_x(g(x), \ell_2(x))$).  One sees from this definition that
$\bar{\ell}_i$ is a comparison point for $\ell_i$ for $i = 1,2$.

Now, since $(\Matx, d_x)$ is CAT(0), we may consider $\triangle g(x)
h(x) k(x)$ and its comparison triangle $\triangle \bar{g}(x)
\bar{h}(x) \bar{k}(x)$ to see that $d_x(\ell_1(x), \ell_2(x)) \leq
\abs{\bar{\ell}_2(x) - \bar{\ell}_1(x)}$ for all $x \in M$.
Therefore,
\begin{align*}
  d(\ell_1, \ell_2) &= \Omega_2(\ell_1, \ell_2) =
  \left(
    \integral{M}{}{d_x(\ell_1(x), \ell_2(x))^2}{d \mu}
  \right)^{1/2} \\
  &\leq
  \left(
    \integral{M}{}{\abs{\bar{\ell}_2(x) - \bar{\ell}_1(x)}^2}{d \mu}
  \right)^{1/2} = \norm{\bar{\ell}_2 - \bar{\ell}_1}_{L^2(M \times
    \Etwo, \mu)}.
\end{align*}
Thus, $\triangle g h k$ satisfies the CAT(0) inequality with respect
to the comparison triangle $\triangle \bar{g} \bar{h} \bar{k}$.  Since
the points $g, h, k \in \overline{\M}$ were arbitrary, we have proved
the following.

\begin{theorem}\label{thm:14}
  The metric completion of $\M$ with respect to its $L^2$ Riemannian
  metric, $\overline{(\M, d)}$, is a CAT(0) space.

  In particular, if $\triangle g h k \subset \M$ is any geodesic
  triangle consisting of smooth metrics, then this triangle satisfies
  the CAT(0) inequality.
\end{theorem}

Note that even though $\overline{\M}$ is a CAT(0) space, we cannot
infer that $\M$ is a space of nonpositive curvature in the sense of
Alexandrov---i.e., locally CAT(0)---as we could in finite dimensions.
This is because, as one sees from Theorem \ref{thm:3}, there is no
$d$-ball of positive radius around any point such that geodesics
between points in this ball exist and remain in $\M$.

\section{Outlook}
\label{sec:outlook}

There are a number of open questions concerning the $L^2$ metric.  As
mentioned in the introduction, one may consider harmonic maps with
$(\M, d)$ or its completion as a target, or consider groups of
isometries of $\M$.  Another is to find submanifolds of $\M$ on which
convergence in the $L^2$ metric implies a more synthetic-geometric
notion of convergence---e.g., Gromov--Hausdorff convergence, or
convergence as a metric-measure space
\cite[Sect.~3$\frac{1}{2}$]{gromov-metric-2007}.  On a related note,
we point out that in \cite[Cor.~6.7]{clarke11:_ricci_kahler}, it has been shown
that if the Kähler-Ricci flow converges w.r.t.~$d$, then it converges
smoothly---so at least for these special paths, convergence in the
$L^2$ metric already implies stronger convergence.

Other open questions that seem difficult to solve concern the moduli
space of Riemannian metrics, also called the space of Riemannian
structures or, sometimes, superspace.  If $\D$ is the group of
orientation-preserving diffeomorphisms of $M$, then this space is the
quotient $\M / \D$, where $\D$ of course acts by pulling back metrics.
It is not hard to see that the $L^2$ metric is invariant under this
action, and so it projects to the quotient, which is a stratified
space \cite{bourguignon75:_une_strat_de_lespac_des_struc_rieman}.

From the metric-geometric standpoint, the first question one must ask
is whether the $L^2$ metric induces a metric space structure on the
quotient.  Because the $L^2$ metric is weak and the quotient is
singular, we cannot \emph{a priori} exclude the situation that two
orbits of the diffeomorphism group become arbitrarily close to one
another.  The very weak nature of convergence with respect to the
$L^2$ metric makes finding a lower bound for the distance between
orbits, either directly or through some geometric invariants, a
difficult proposition (cf.~\cite[Sect.~4.3]{Clarke:2010ur} for more on
this).

Assuming the previous question is answered in the affirmative, another
question one could ask about the moduli space is what its completion
with respect to the $L^2$ metric is.  Potentially, some of the
very pathological degenerations that lead to losing all regularity of
a limit metric in the completion of $(\M, d)$ come from degenerations
along a diffeomorphism orbit.

\bibliography{main,papers}
\bibliographystyle{hamsalpha}

\end{document}